\documentclass[11pt]{article}
\usepackage{amsfonts}

\usepackage{graphics}
\usepackage{indentfirst}
\usepackage{cite}
\usepackage{latexsym}
\usepackage{amsmath,amsthm}
\usepackage{amssymb}
\usepackage[dvips]{epsfig}
\usepackage{amscd}
\usepackage{mathrsfs}
\usepackage{color}
\newtheorem{theorem}{Theorem}[section]
\newtheorem{remark}{Remark}[section]

\newtheorem{definition}{Definition}[section]
\newtheorem{lemma}[theorem]{Lemma}

\newcommand{\n}{\rho}

\newcommand{\lm}{\lambda}

\renewcommand{\div}{ {\rm div }  }

\newcommand{\na}{\nabla }

\newcommand{\pa}{\partial}

\newcommand{\bt}{\begin{theorem}}
\newcommand{\bl}{\begin{lemma}}
\newcommand{\el}{\end{lemma}}
\newcommand{\et}{\end{theorem}}
\newcommand{\ga}{\gamma}

\newcommand{\OM}{\Omega}

\newcommand{\curl}{{\rm curl} }

\newcommand{\de}{\delta}

\newcommand{\la}{\label}
\newcommand{\si}{\sigma}

\newcommand{\ol}{\overline}

\newcommand{\bn}{\begin{eqnarray}}
\newcommand{\en}{\end{eqnarray}}
\newcommand{\bnn}{\begin{eqnarray*}}
\newcommand{\enn}{\end{eqnarray*}}

\newcommand{\bnnn}{\begin{eqnarray*}}
\newcommand{\ennn}{\end{eqnarray*}}
\newcommand{\ben}{\begin{enumerate}}
\newcommand{\een}{\end{enumerate}}

\newcommand{\rs}{\rho^*}

\newcommand{\T}{\mathbb{T}}

\newcommand{\ba}{\begin{aligned}}
\newcommand{\ea}{\end{aligned}}
\newcommand{\be}{\begin{equation}}
\newcommand{\ee}{\end{equation}}

\def\p{\partial}
\def\norm[#1]#2{\|#2\|_{#1}}

\def\lam{\lambda}
\def\ep{\varepsilon}

\def\r{\mathbb{R}}
\def\rr{\mathbb{R}^2}
\def\rrr{\mathbb{R}^3}

\makeatletter      % '@' is now a normal "letter" for TeX
\@addtoreset{equation}{section}
\makeatother       % '@' is restored as a "non-letter" character for TeX

\title{Global Existence and Incompressible Limit for the Three-Dimensional Axisymmetric Compressible Navier-Stokes Equations with Large Bulk Viscosity and Large Initial Data}

\date{}

\author{$\text{Qinghao L{\small EI}}^{a,b}\thanks{Email addresses:  leiqinghao22@mails.ucas.ac.cn (Q. H. Lei)}$ \\
a. School of Mathematical Sciences,\\ University of Chinese Academy of Sciences,
Beijing 100049, P. R. China;\\
b. Institute of Applied Mathematics,\\ Academy of Mathematics and Systems Science, \\
Chinese Academy of Sciences, Beijing 100190, P. R. China}

\begin{document}
\maketitle

\begin{abstract}
In this paper, we study the three-dimensional axisymmetric compressible Navier-Stokes equations with slip boundary conditions in a cylindrical domain excluding the axis.
We establish the global existence and exponential decay of weak, strong, and classical solutions with large initial data and vacuum,
under the assumption that the bulk viscosity coefficient is sufficiently large.
Moreover, we prove that as the bulk viscosity coefficient tends to infinity,
the solutions of the compressible Navier-Stokes equations converge to those of the inhomogeneous incompressible Navier-Stokes equations.
\end{abstract}

\section{Introduction and main results}
We consider the three-dimensional barotropic compressible
Navier-Stokes equations which read as follows:
\be\ba\la{ns}
\begin{cases}
\rho_t + \div(\rho \mathbf{u}) = 0,\\
(\n \mathbf{u})_t + \div(\n \mathbf{u}\otimes \mathbf{u}) -\mu \Delta \mathbf{u} 
- (\mu + \lm)\na\div \mathbf{u} + \na P = 0,
\end{cases}
\ea\ee
where $t \ge 0$ is time and $x \in \OM \subset \rrr$ is the spatial coordinate.
The unknown functions $\n=\n(x,t)$ and $\mathbf{u}(x,t)=(u^1(x,t),u^2(x,t),u^3(x,t))$ represent the 
density and velocity of the fluid, respectively.
The pressure $P$ is given by
\be\ba\la{i1}
  P=a\n^\ga,
\ea\ee
with constants $a>0$ and $\ga > 1$. Without loss of generality, we set $a=1$. The shear viscosity coefficient $\mu$ and the bulk viscosity coefficient $\lam$ satisfy
the physical restrictions:
\be\la{i2}
\mu>0,\quad \mu+\lam\geq 0.
\ee
For later use, we define
\be\ba\la{nu}
\nu \triangleq 2\mu + \lam.
\ea\ee
The system is supplemented with the initial data
\be\la{i30}
\n(x,0)=\n_0(x),\quad \n \mathbf{u}(x,0)= \n_0 \mathbf{u}_0(x), \quad x\in \OM,
\ee
and the slip boundary conditions:
\be\la{i3}\ba
\mathbf{u} \cdot n = 0,\ \ \curl \mathbf{u} \times n = -K \mathbf{u} \,\,\,\text{on}\,\,\, \partial\Omega,
\ea\ee
where $K=K(x)$ is a $3 \times 3$ symmetric matrix defined on $\p \OM$,
$n=(n_1,n_2,n_3)$ denotes the unit outer normal vector to the boundary $\partial \Omega$.

Considerable research has been devoted to the theory of weak and classical
solutions for the multidimensional compressible Navier-Stokes equations with constant viscosity coefficients.
For initial density away from vacuum, the local existence and uniqueness of classical solutions were established by Nash \cite{N} and Serrin \cite{S}, respectively.
Subsequently, for initial density that contains vacuum and may vanish on open sets,
the local existence and uniqueness of strong solutions were proved in \cite{CCK,CK,CK2,SS} 
and references therein.
The first result of global classical solutions was obtained by
Matsumura--Nishida\cite{MN1} under the assumption that the initial data are close
to a non-vacuum equilibrium in some Sobolev space $H^s$.
Thereafter, Hoff \cite{H1,H2,H3} investigated the problem for discontinuous
initial data and developed new a priori estimates for the material derivative $\dot{\mathbf{u}}$.
Regarding the existence of weak solutions for arbitrary data, the major breakthrough was due to Lions \cite{L2},
who proved the global existence of weak solutions
when the initial energy is finite and the adiabatic exponent $\ga$ is suitably large.
For example, in the three-dimensional case, the exponent was required to satisfy
$\ga \ge \frac{9}{5}$, which was later relaxed to $\ga > \frac{3}{2}$ by
Feireisl et al. \cite{FNP}.
Recent work by Huang-Li-Xin \cite{HLX2} and Li-Xin \cite{LX2} has established the global existence and uniqueness of classical solutions to (\ref{ns})
for three-dimensional and two-dimensional Cauchy problems, respectively,
provided that the initial energy is suitably small,
while allowing for possibly large oscillations and initial density that may contain vacuum and even have compact support.
Subsequently, Cai-Li \cite{CL} extended these results to three-dimensional bounded domains,
where the velocity field is subject to the slip boundary conditions.

Recently, for the system (\ref{ns})--(\ref{i30}) in two-dimensional periodic domains,
Danchin-Mucha \cite{DM} obtained the global existence of weak solutions
when the bulk viscosity is sufficiently large and $\nu^{1/2} \|\div \mathbf{u}_0\|_{L^2} \le M$, where $M$ is a fixed positive constant.
They also proved that as the bulk viscosity tends to infinity,
the weak solutions converge to solutions of the
inhomogeneous incompressible Navier-Stokes equations.
Very recently, Lei-Xiong \cite{LX3} generalized their results by removing
the condition that $\nu^{1/2} \|\div \mathbf{u}_0\|_{L^2} \le M$.
Furthermore, building upon the pointwise representation of the effective viscous flux
in general two-dimensional bounded simply connected domains developed by
Fan-Li-Li \cite{FLL},
Lei-Xiong \cite{LX4} extended the results of \cite{DM} to such domains.

The aim of this paper is to study the global existence and incompressible limit 
of weak, strong, and classical solutions for the three-dimensional axisymmetric
compressible Navier-Stokes equations in a cylindrical domain excluding the axis, under slip boundary conditions.
For simplicity, the domain is defined as
\be\la{om}\ba
\OM=A \times \T,
\ea\ee
where $A=\{ (x_1,x_2) \in \rr : 1<x^2_1+x^2_2<4 \}$ is a two-dimensional annulus, and $\T=\r / \mathbb{Z}$ is the one-dimensional torus.
The flow is assumed to be periodic in the $x_3$-direction with period 1.

For $(x_1,x_2,x_3)\in \rrr$, we introduce the cylindrical coordinate transformation
\be\ba\nonumber
\begin{cases}
x_1=r \cos \theta,\\
x_2=r \sin \theta,\\
x_3=z,\\
\end{cases}
\ea\ee
and define the standard orthonormal basis in $\rrr$ as:
\be\ba\nonumber
\mathbf{e}_r=\frac{(x_1,x_2,0)}{r},
\ \mathbf{e}_\theta=\frac{(-x_2,x_1,0)}{r},
\  \mathbf{e}_z=(0,0,1).
\ea\ee
where $r=\sqrt{x^2_1+x^2_2}$.

A scalar function $g$ or a vector-valued function
$\mathbf{f}=f_r \mathbf{e}_r + f_\theta \mathbf{e}_\theta + f_z\mathbf{e}_z$
is called axisymmetric if $g$, $f_r$, $f_\theta$ and $f_z$ do not depend on $\theta$.

We investigate axisymmetric solutions to the system (\ref{ns})--(\ref{i3})
that are periodic in $x_3$ with period $1$.
More precisely, we consider solutions of the form:
\be\la{sdc}\ba
\begin{cases}
\n(x_1,x_2,x_3,t)=\n(r,z,t), \\
\mathbf{u}(x_1,x_2,x_3,t)=u_r(r,z,t)\mathbf{e}_r + u_\theta(r,z,t)\mathbf{e}_\theta + u_z(r,z,t)\mathbf{e}_z, \\
\n(x_1,x_2,x_3+1,t) = \n(x_1,x_2,x_3,t), \ 
\mathbf{u}(x_1,x_2,x_3+1,t)=\mathbf{u}(x_1,x_2,x_3,t),
\end{cases}
\ea\ee
for any $(x_1,x_2)\in A$ and $x_3 \in \r$.

Before stating the main results, we first explain the notations
and conventions used throughout this paper. We set
\be\ba\nonumber
  \int f dx = \int_{\OM} fdx,\quad
   \ol{f}=\frac{1}{|\OM|}\int fdx.
\ea\ee
For an integer $s \ge 1$ and $1\leq r\leq\infty$, the standard Lebesgue and Sobolev spaces are denoted as follows:
\be\ba\nonumber
\begin{cases}
L^r =L^r(\OM),\quad W^{s,r} =W^{s,r}(\OM),\quad H^s =W^{s,2}, \\
\tilde{H}^1 =\{v \in H^1(\Omega )\vert v \cdot n=0,\  \mathrm{curl} v \times n =-K v \text{ on } \partial \Omega \}.
\end{cases}
\ea\ee

In the axisymmetric setting, the corresponding two-dimensional domain $D$ associated with $\OM$ is defined, via cylindrical coordinate transformations, as
\be\la{ewD}\ba
D=\{ (r,z)\in \rr:1<r<2,0<z<1 \}.
\ea\ee
The shear stress tensor is given by
\be\ba\nonumber
D(v)=\frac{1}{2} \left( \na v + (\na v)^{ \text{tr} } \right).
\ea\ee
The material derivative is defined as
\be\ba\nonumber
\frac{D}{Dt}f=\dot{f} \triangleq f_t + \mathbf{u}\cdot\na f.
\ea\ee
We denote the initial total energy as
\be\ba\nonumber
E_0 \triangleq \int \left( \frac{1}{2} \rho_0 |\mathbf{u}_0|^2 + \frac{1}{\ga-1}\n^\ga_0 \right) dx.
\ea\ee

We now introduce the definitions of weak and strong solutions in the axisymmetric class for system (\ref{ns}).
\begin{definition}
A pair $(\n,\mathbf{u})$ is called a weak solution in the axisymmetric class
to \eqref{ns} if it is axisymmetric and periodic in $x_3$ with period $1$
(i.e., \eqref{sdc} holds) and satisfies \eqref{ns} in the sense of distribution.

Furthermore, a weak solution in the axisymmetric class is called a strong solution in the axisymmetric class
if all derivatives involved in \eqref{ns} are regular distributions
and equations \eqref{ns} hold almost everywhere in
$\OM\times(0,T)$.
\end{definition}

We state the first main result concerning the global existence and exponential decay of weak solutions as follows:
\begin{theorem}\la{th0}
Let $K$ be a smooth and symmetric $3 \times 3$ axisymmetric matrix-valued function such that $K + 2D(n)$ is a positive semi-definite symmetric matrix,
and $K + 2D(n)$ is positive definite on some $\Sigma \subset \p \OM$ with $|\Sigma|>0$.
Assume that the initial data $(\n_0,\mathbf{u}_0)$ satisfy (\ref{sdc}) and
\be\la{wsol1}\ba
0\le \n_0 \in L^\infty,\quad \mathbf{u}_0 \in \tilde{H}^1.
\ea\ee
Then, there exists a positive constant $\nu_1$ depending only on
$\ga$, $\mu$, $\|\n_0\|_{L^\infty}$, $\| \mathbf{u}_0 \|_{H^1}$, and $K$,
such that when $\nu \ge \nu_1 $, the problem \eqref{ns}--\eqref{i3} has at least one weak solution $(\n,\mathbf{u})$ within the axisymmetric class
on $\OM \times (0,\infty)$ satisfying
\be\la{wsol2}\ba
0\le \n(x,t) \leq 2 \left( 1+\|\n_0\|_{L^\infty} \right) e^{ \frac{\ga-1}{\ga |\OM|} E_0 },
\quad \mathrm{for\ any\ }(x,t)\in \OM \times[0,\infty),
\ea\ee
and
\be\la{wsol3}\ba
\begin{cases}
\rho\in L^{\infty}(\OM \times (0,\infty)) \cap C([0,\infty);L^p), \\ 
\mathbf{u} \in L^2(0,\infty;H^1) \cap L^\infty(0,\infty;H^1), \\
t^{1/2}\mathbf{u}_t \in L^2(0,T;L^2), t^{1/2} \na \mathbf{u} \in L^\infty(0,T;L^p),
\end{cases}
\ea\ee
for any $0<T<\infty$ and $1 \le p < \infty $.

Moreover, for any $s\in [1,\infty)$, there exist positive constants $C$, $D_0$ and $\nu_2$ depending only on
$s$, $\ga$, $\mu$, $\| \n_0 \|_{L^1 \cap L^\infty}$, $\| \mathbf{u}_0 \|_{H^1}$, and $K$, such that
if $\nu \ge \nu_2$, then for any $t\ge 1$,
\be\la{wsol4}\ba
\| \n-\ol{\n_0}\|_{L^s} + \| \na \mathbf{u} \|_{L^s}
+ \| \sqrt{\n} \dot{\mathbf{u}} \|_{L^2} \le C e^{-\alpha_0 t},
\ea\ee
where $\alpha_0=\frac{D_0}{\nu}$.
\end{theorem}

The second result addresses the singular limit from the axisymmetric compressible Navier-Stokes equations to the axisymmetric incompressible Navier-Stokes equations.

\begin{theorem}\la{th01}
Under the conditions of Theorem \ref{th0}, assume further that $\div \mathbf{u}_0 = 0$.
Fix $\mu>0$, and let $\nu_1$ be as in Theorem \ref{th0}.
For $\nu \ge \nu_1$,
let $(\n^{\nu},\mathbf{u}^{\nu})$ be the global weak solution of \eqref{ns}--\eqref{i3} given by Theorem \ref{th0}.
Then, as $\nu$ tends to $\infty$, $(\n^{\nu},\mathbf{u}^{\nu})$ has a subsequence that
converges to a global solution of the following 
inhomogeneous incompressible Navier-Stokes equations:
\be\la{isol1}\ba
\begin{cases}
\rho_t + \div(\rho \mathbf{u}) = 0,\\
(\n \mathbf{u})_t + \div(\n \mathbf{u}\otimes \mathbf{u}) -\mu \Delta \mathbf{u}
+ \na \pi = 0, \\
\div \mathbf{u}=0, \\
\mathbf{u} \cdot n = 0, \quad \curl \mathbf{u} \times n=-K \mathbf{u} \quad \textup{on}\,\,\, \partial\Omega, \\
\end{cases} 
\ea\ee
with initial data $\n(\cdot,0)=\n_0,\ \n \mathbf{u}(\cdot,0)=\n_0 \mathbf{u}_0$,
and $(\n,\mathbf{u})$ satisfies
\be\la{isol2}\ba
\begin{cases}
\rho\in L^{\infty}(\OM \times (0,\infty)) \cap C([0,\infty);L^p), \\ 
\mathbf{u} \in L^2(0,\infty;H^2) \cap L^\infty(0,\infty;H^1), \\
\pi \in L^2(0,\infty;H^1),
\end{cases} 
\ea\ee
for any $0 <T<\infty$ and $1\le p <\infty$.
Moreover, we have 
\be\la{isol3}\ba
\div \mathbf{u}^{\nu} = O(\nu^{-1/2}) \  in \  L^2(\OM \times (0,\infty)) \cap L^\infty(0,\infty;L^2).
\ea\ee
Additionally, if the initial data $(\n_0,\mathbf{u}_0)$ further satisfy for some $q>2$,
\be\la{insc1}\ba
0\le \n_0 \in W^{1,q},\quad \mathbf{u}_0 \in \tilde{H}^1,\quad \div \mathbf{u}_0=0,
\ea\ee
then the entire sequence $(\n^\nu,\mathbf{u}^\nu)$ converges to the unique global strong solution of (\ref{isol1}) and satisfies for any $0<T<\infty$,
\be\la{insc3}\ba
\begin{cases}
\rho \in C([0,T];W^{1,q} ), \quad \n_t\in L^\infty(0,T;L^2), \\ 
\mathbf{u} \in L^\infty(0,T; H^1) \cap L^{(q+1)/q}(0,T; W^{2,q}), \\ 
t^{1/2} \mathbf{u} \in L^2(0,T; W^{2,q}) \cap L^\infty(0,T;H^2), \\
t^{1/2} \mathbf{u}_t,\ \pi \in L^2(0,T;H^1), \quad t^{1/2} \pi \in L^\infty(0,T;H^1), \\
\sqrt{\n} \mathbf{u}_t \in L^2(\OM \times(0,T)).
\end{cases}
\ea\ee
\end{theorem}

The following result establishes the global existence and exponential decay of strong solutions.

\begin{theorem}\la{th1}
Assume that the conditions of Theorem \ref{th0} hold, with (\ref{wsol1}) replaced by
\be\la{ssol1}\ba
0\le \n_0 \in W^{1,q},\quad \mathbf{u}_0 \in \tilde{H}^1,
\ea\ee
for some $q>3$.
Then, for the same $\nu_1$ as in Theorem \ref{th0}, when $\nu \ge \nu_1 $,
the problem \eqref{ns}--\eqref{i3} has a unique strong solution $(\n,\mathbf{u})$
within the axisymmetric class on $\OM \times (0,\infty)$ satisfying (\ref{wsol2}) and for any $0<T<\infty$,
\be\la{ssol4}\ba
\begin{cases}
\rho\in C([0,T];W^{1,q} ), \quad \n_t\in L^\infty(0,T;L^2), \\ 
\mathbf{u} \in L^\infty(0,T; H^1) \cap L^{(q+1)/q}(0,T; W^{2,q}), \\ 
t^{1/2} \mathbf{u} \in L^2(0,T; W^{2,q}) \cap L^\infty(0,T;H^2), \\
t^{1/2} \mathbf{u}_t \in L^2(0,T;H^1), \\
\n \mathbf{u} \in C([0,T];L^2), \quad \sqrt{\n} \mathbf{u}_t\in L^2(\OM \times(0,T)).
\end{cases} 
\ea\ee
Moreover, the strong solution $(\n,\mathbf{u})$ satisfies the exponential decay estimates (\ref{wsol4}).
\end{theorem}

If the initial data $(\n_0,\mathbf{u}_0)$ satisfy higher regularity and compatibility conditions, we can obtain the global existence and exponential decay of classical solutions to (\ref{ns}).

\begin{theorem}\la{th2}
In addition to the assumptions in Theorem \ref{th0}, assume further that $(\n_0,\mathbf{u}_0)$ satisfy for some $q>2$,
\be\la{csol1}\ba
0\le \n_0 \in W^{2,q},\quad \mathbf{u}_0 \in H^2 \cap \tilde{H}^1,
\ea\ee
and the following compatibility condition:
\be\la{csol2}\ba
- \mu \Delta \mathbf{u}_0 - (\mu + \lm )\nabla \div \mathbf{u}_0 +  \nabla P(\n_0)=\n_0^{1/2}g,
\ea\ee
for some $g\in L^2$.
Then, for the same $\nu_1$ as in Theorem \ref{th0}, when $\nu \ge \nu_1 $,
the problem \eqref{ns}--\eqref{i3} has a unique classical solution $(\n,\mathbf{u})$
within the axisymmetric class on $\OM \times (0,\infty)$ satisfying (\ref{wsol2}) and for any $0<T<\infty$,
\be\la{csol4}\ba
\begin{cases}
(\rho,P(\n))\in C([0,T];W^{2,q} ),\quad  (\n_t,P_t)\in L^\infty(0,T;H^1), \\ 
(\n_{tt},P_{tt})\in L^2(0,T;L^2), \\
\mathbf{u} \in L^\infty(0,T; H^2) \cap L^2(0,T;H^3), \quad \mathbf{u}_t \in L^2(0,T; H^1) \\ 
\na \mathbf{u}_t, \na^3\mathbf{u} \in L^{(q+1)/q}(0,T;L^q), \\
t^{1/2}\na^3 \mathbf{u} \in L^\infty(0,T;L^2)\cap L^2(0,T; L^q), \\
t^{1/2} \mathbf{u}_t\in L^\infty(0,T;H^1) \cap L^2(0,T;H^2), \\ 
t^{1/2}\na^2(\n \mathbf{u})\in L^\infty(0,T;L^q),\quad \n^{1/2} \mathbf{u}_t\in L^{\infty}(0,T;L^2), \\
t\n^{1/2} \mathbf{u}_{tt},\quad t\na^2 \mathbf{u}_t \in L^{\infty}(0,T;L^2), \\
t\na^3 \mathbf{u} \in L^{\infty}(0,T;L^q),\quad t\na \mathbf{u}_{tt} \in L^2(0,T;L^2).
\end{cases} 
\ea\ee
Furthermore, the classical solution $(\n,\mathbf{u})$ satisfies the exponential decay estimates (\ref{wsol4}).
\end{theorem}

Finally, following \cite{CL,LX}, the exponential decay estimates (\ref{wsol4}) yield the following large-time behavior of the spatial gradient of the density for the strong solution in Theorem \ref{th1} when the initial density contains vacuum.

\begin{theorem}\la{th3}
In addition to the assumptions in Theorem \ref{th1}, 
we further assume that there exists some point $x_0\in \OM$ such that $\n_0(x_0)=0$. 
Then for any $r>2$, there exists a positive constant $C$ depending only on
$r$, $\mu$, $\ga$, $\|\n_0\|_{L^1 \cap L^\infty}$, $\| \mathbf{u}_0 \|_{H^1}$, and $K$ such that for any $t \ge 1$,
\be\la{pbu0}\ba
\| \na \n(\cdot ,t) \|_{L^r} \ge C e^{\alpha_0 \frac{r-2}{r} t }.
\ea\ee
\end{theorem}

\begin{remark}\la{lrk1}
Since $q>2$, it follows from (\ref{sdc}) and (\ref{csol4}) that
\be\la{csol5}\ba
\n, P(\n) \in C([0,T];W^{2,q})\hookrightarrow C\left([0,T];C^1(\overline{\OM}) \right).
\ea\ee
Moreover, for any $0<\tau<T<\infty$, the standard embeddings yield
\be\la{csol6}\ba
\mathbf{u} \in L^\infty(\tau ,T;W^{3,q})\cap H^1(\tau ,T;H^2)\hookrightarrow
C\left([\tau,T];C^2(\overline{\OM}) \right),
\ea\ee
and	
\be\la{csol7}\ba
\mathbf{u}_t\in L^\infty(\tau ,T;H^2)\cap H^1(\tau ,T;H^1)\hookrightarrow
C\left([\tau,T]; C(\overline{\OM}) \right).
\ea\ee
By virtue of $(\ref{ns})_1$, (\ref{csol5}), and (\ref{csol6}), we have
\be\la{csol8}\ba
\n_t = -\n \div \mathbf{u} - \mathbf{u} \cdot \na \n \in C(\overline{\OM} \times [\tau,T]).
\ea\ee
Hence, the solution obtained in Theorem \ref{th2} is in fact a classical solution to the problem \eqref{ns}--\eqref{i3} in $\OM \times (0,\infty)$.
\end{remark}

\begin{remark}\la{lrk2}
It is worth noting that in our existence Theorems \ref{th0}, \ref{th1},
and \ref{th2}, we only require the bulk viscosity coefficient to be sufficiently large,
without imposing any restrictions on the size of the initial data.
These results therefore establish the global existence for the three-dimensional axisymmetric 
compressible Navier-Stokes equations with large initial data.
\end{remark}

We now make some comments on the analysis of this paper.
First, we note that for initial data satisfying (\ref{csol1}) and (\ref{csol2}), the local existence and uniqueness of classical solutions to the problem (\ref{ns})--(\ref{i3})
can be established following methods similar to \cite{H2021}.
To extend the classical solution globally in time,
it suffices to establish global a priori estimates for classical solutions to (\ref{ns})--(\ref{i3})
in appropriate high-order norms.
The key issue is to obtain the time-uniform upper bound of the density.
Owing to the axisymmetry and the fact that the domain excludes the axis, we have the Gagliardo-Nirenberg inequalities (\ref{ewgn01}) analogous to the two-dimensional case.
These inequalities are crucial in our subsequent analysis,
particularly in estimating $\| \na \mathbf{u}\|_{L^p}$ for $2 \le p <\infty$.
Using the div-curl type estimates in Lemmas \ref{dltdc1} and \ref{dltdc2}, we first obtain the standard energy estimate (\ref{zdc1}).
Then, exploiting that the domain is away from the axis and applying the two-dimensional logarithmic interpolation inequality (\ref{logbd1}), we show that the $L^\infty(0,T;L^2(\OM))$-norm of $\na \mathbf{u}$ is bounded by a polynomial function of $\nu$ (see (\ref{zdc04})).
To obtain the upper bound of the density, the crucial step is to estimate the $L^\infty$-norm of the effective viscous flux $G$ (see (\ref{gw}) for its definition).
The slip boundary conditions and $(\ref{ns})_2$ imply that $G$ satisfies the elliptic equation (\ref{evf1}).
By axisymmetry, we transform equation (\ref{evf1}) into its two-dimensional form (\ref{evf2}).
Subsequently, using Green's function for the two-dimensional unit disk and a conformal mapping, we derive the pointwise estimate of $G$ (see Lemma \ref{l6}).
Combining this pointwise estimate with the a priori bounds
and adapting the framework of \cite{DM,LX4}, we obtain the upper bound of the density, provided the bulk viscosity coefficient $\nu$
is sufficiently large (see Lemma \ref{l7} and its proof).
Following the approach in \cite{HLX2,CL,LX4}, we then establish exponential decay and the necessary higher-order derivative estimates,
thus allowing us to extend the local solution globally.
It is noteworthy that in dealing with the boundary terms,
we have used two key observations from \cite{CL}.
Specifically, the boundary condition $\mathbf{u} \cdot n = 0$ on $\p \OM$ leads to the following identities:
\be\la{bjds}\ba
\mathbf{u} = -(\mathbf{u} \times n) \times n \triangleq \mathbf{u}^\bot \times n,
\quad (\mathbf{u} \cdot \na) \mathbf{u} \cdot n=-(\mathbf{u} \cdot \na) n \cdot \mathbf{u}.
\ea\ee

Finally, we establish the singular limit from the axisymmetric compressible Navier-Stokes equations to the axisymmetric inhomogeneous incompressible Navier-Stokes equations.
From the estimates (\ref{zdc04}) and (\ref{gb002}), we obtain that the solution family $(\rho^\nu,\mathbf{u}^\nu)$ has uniform bounds independent of $\nu$, under the condition $\div \mathbf{u}_0 =0$.
Combining these uniform bounds with the standard compactness arguments as in \cite{LX3}, we prove that as the bulk viscosity tends to infinity,
a subsequence of $(\rho^\nu,\mathbf{u}^\nu)$ converges to a solution of the axisymmetric inhomogeneous incompressible Navier-Stokes equations.
Furthermore, building on the global well-posedness results for strong solutions to the three-dimensional axisymmetric inhomogeneous incompressible Navier-Stokes equations with large initial data established in \cite{GWX,WG}, 
we employ a weak-strong uniqueness argument to show that
when the initial data satisfy the condition (\ref{insc1}),
the entire family $(\rho^\nu,\mathbf{u}^\nu)$ converges to
the unique global strong solution of the axisymmetric inhomogeneous incompressible Navier-Stokes equations.

The rest of this paper is organized as follows:
Section 2 presents essential preliminary results and key inequalities.
Sections 3 and 4 are devoted to deriving the necessary a priori estimates.
Finally, Section 5 combines these a priori estimates to prove our main results, Theorems~\ref{th0}--\ref{th3}.

\section{Preliminaries}
In this section, we recall some known facts and elementary inequalities
that are used frequently in the subsequent analysis.

First, we present the following local existence result of classical solutions, which can be proved similarly to \cite[Theorem~1.4]{H2021}.
\begin{lemma}\la{lct}
Assume that the initial data $\left( \n_0,\mathbf{u}_0 \right)$ satisfies (\ref{csol1}) and (\ref{csol2}).
Then there is a small time $T>0$ depending only on $\mu$, $\lambda$, $\ga$, $q$, $\| \n_0\|_{W^{2,q}}$, $\|\mathbf{u}_0\|_{H^2}$, and $\|g\|_{L^2}$,
such that there exists a unique classical solution $(\n,\mathbf{u})$ to the problem \eqref{ns}--\eqref{i3}
in $\OM \times (0,T]$ satisfying (\ref{csol4}).
\end{lemma}

Owing to the rotation and transformation invariance of the system (\ref{ns})--(\ref{i3}) and following the arguments in \cite[Lemma 2]{GWX}, we have:
\begin{lemma}\la{dcx}
Assume that the initial data is axisymmetric and periodic in $x_3$ with period $1$.
Then the local classical solution of \eqref{ns}--\eqref{i3} is also axisymmetric and periodic in $x_3$ with period $1$.
\end{lemma}

For three-dimensional axisymmetric functions, the following Gagliardo-Nirenberg inequalities can be found in \cite[Lemma 2.4]{LKV}.
\begin{lemma}\la{ewgn}
Let $\OM$ be as in $(\ref{om})$,
and let $\mathbf{f}$ and $g$ be axisymmetric vector-valued and scalar functions on $\OM$, respectively.
Then, for any $p\in [2,\infty)$, there exists a generic constant $C>0$ such that
\be\la{ewgn01}\ba
\| \mathbf{f} \|_{L^p}
\le C p^{1/2} \| \mathbf{f} \|^{\frac{2}{p}}_{L^2} \| \mathbf{f} \|^{1-\frac{2}{p}}_{H^1},\quad
\| g \|_{L^p} \le C p^{1/2} \| g \|^{\frac{2}{p}}_{L^2} \| g \|^{1-\frac{2}{p}}_{H^1}.
\ea\ee
\end{lemma}

The following Poincar\'e type inequality can be found in \cite{F}.
\begin{lemma}\la{pt}
Let $v\in H^1$, and let $\n$ be a non-negative function satisfying
\be\ba\la{pt0}
0<M_1\leq \int \n dx,\quad \int \n^r dx \leq M_2,
\ea\ee
with $r>1$. Then there exists a positive constant $C$ depending only on 
$M_1$, $M_2$, and $r$ such that
\be\ba\la{pt1}
\|v\|_{L^2}^2 \leq C\int \n |v|^2 dx + C \|\na v\|_{L^2}^2.
\ea\ee
\end{lemma}

The following div-curl estimates will be used frequently in the subsequent analysis (see \cite{AJ,WWV}).
\begin{lemma}\la{dc}
Let $k \ge 0$ be an integer and $1<p<\infty$.
Assume that $\OM$ is a bounded domain in $\rrr$ and its $C^{k+1,1}$ boundary $\p \OM$ only has a finite number of 2-dimensional connected components.
Then, for $v \in W^{k+1,p}(\OM)$ with $(v \cdot n)|_{\p \OM} = 0$ or $(v \times n)|_{\p \OM} = 0$,
there exists a positive constant $C$ depending only on $k$, $p$, and $\OM$ such that
\be\ba\la{dc1}
\| v \|_{W^{k+1,p}(\OM)} \le C \left( \|\div v\|_{W^{k,p}(\OM)} + \| \curl v \|_{W^{k,p}(\OM)} + \| v \|_{L^p(\OM)} \right).
\ea\ee
\end{lemma}

More generally, with the help of Lemma \ref{dc} and following an approach analogous to that in \cite{FLL,FLW},
we derive the following weighted div-curl estimates.
\begin{lemma}\la{wdc}
Let $\OM$ be as in Lemma \ref{dc}.
For any $v \in H^2(\OM)$ with $(v \cdot n)|_{\p \OM} = 0$, there exist positive constants $C$ and $\hat{\de}$
both depending only on $\OM$ such that for any $\de \in (0,\hat{\de})$,
\be\ba\la{wdc1}
\int_\OM |v|^{\de} |\na v|^2 dx
\le C \int_\OM |v|^{\de} \left( (\div v)^2 + |\curl v|^2 + |v|^{2} \right) dx.
\ea\ee
\end{lemma}
\begin{proof}
First, using Cauchy's inequality, we directly calculate that
\be\la{wdc01}\ba
\left( \div ( |v|^{\frac{\de}{2}} v ) \right)^2 \le 2 |v|^{\de} (\div v)^2
+ \de^2 |v|^{\de} |\na v|^2, \\
\left| \curl ( |v|^{\frac{\de}{2}} v ) \right|^2 \le 2 |v|^{\de} |\curl v|^2
+ \de^2 |v|^{\de} |\na v|^2, \\
\left| \na ( |v|^{\frac{\de}{2}} v ) \right|^2 \ge \frac{1}{2} |v|^{\de} |\na v|^2
- \de^2 |v|^{\de} |\na v|^2.
\ea\ee
Observing that $|v|^{\frac{\de}{2}} v \cdot n=0$ on $\p \OM$, by applying (\ref{dc1}), we obtain
\be\la{wdc02}\ba
\int_\OM \left| \na ( |v|^{\frac{\de}{2}} v ) \right|^2 dx
& \le C \int_\OM \left( \left( \div ( |v|^{\frac{\de}{2}} v ) \right)^2
+ \left| \curl ( |v|^{\frac{\de}{2}} v ) \right|^2 + |v|^{\de+2} \right) dx.
\ea\ee
Putting (\ref{wdc01}) into (\ref{wdc02}) and choosing $\hat{\de}>0$ sufficiently small, we obtain (\ref{wdc1}) for all $\de \in (0,\hat{\de})$,
thereby completing the proof.
\end{proof}

The following lemma can be found in \cite[Lemma 2.7]{LKV}.
\begin{lemma}\la{dltdc1}
Let $\OM$ be an axisymmetric and bounded Lipschitz domain in $\rrr$.
Then for $v \in H^1$ with $v \cdot n=0$ on $\p \OM$ and smooth positive semi-definite $3 \times 3$ symmetric matrix $B$ satisfying $B>0$ on some $\Sigma \subset \p \OM$ with $|\Sigma|>0$,
there exists a positive constant $\Lambda$ depending only on $B$ and $\OM$ such that
\be\la{dltdc01}\ba
\| v \|^2_{H^1} \le \Lambda \left( \| D(v) \|^2_{L^2} + \int_{\p \OM} v \cdot B \cdot v ds \right).
\ea\ee
\end{lemma}

As shown in \cite[Lemma 6.2]{CL}, we have the following result.
\begin{lemma}\la{dltdc2}
Let $\OM$ be a smooth bounded domain in $\rrr$. Then for $v \in H^2(\OM)$ with $v \cdot n=0$ on $\p \OM$, the following identity holds:
\be\la{dltdc02}\ba
2\int D(v) \cdot D(v) dx = 2\int (\div v)^2 dx + \int |\curl v|^2 dx -2\int_{\p \OM} v \cdot D(n) \cdot v ds.
\ea\ee
\end{lemma}

Next, the following logarithmic interpolation inequality can be found in \cite{LX4}.
\begin{lemma}\la{logbd}
Let $D$ be a bounded domain in $\rr$ with Lipschitz boundary.
For any $\n \in L^\infty(D)$ with $0 \le \n \le \rs $ and $v \in H^1(D)$,
there exists a positive constant $C$ depending only on $\rs$ and $D$ such that
\be\la{logbd1}\ba
\| \sqrt{\n} v \|^2_{L^4(D)} \le C \| \sqrt{\n} v \|_{L^2(D)} + C \left( 1+\| \sqrt{\n} v \|_{L^2(D)} \right)
\| v \|_{H^1(D)} \log^{\frac{1}{2}} \left( 2+\| v \|^2_{H^1(D)} \right).
\ea\ee
\end{lemma}

The following Beale-Kato-Majda type inequality plays a crucial role in deriving the estimates of $\| \na \mathbf{u}\|_{L^{\infty}}$ and $\| \na \n\|_{L^{q}}$.
It was established in \cite{K} for the case $\div \mathbf{u} \equiv 0$; see also \cite{BKM,CL,HLX1}.
\begin{lemma}\la{bkm}
Let $\OM$ be a bounded domain in $\rrr$ with smooth boundary.
For $3<q<\infty$,
there exists a positive constant $C$ 
depending only on $q$ and $\OM$ such that every function $\mathbf{u} \in \left\{ W^{2, q}(\Omega ) \big| \mathbf{u} \cdot n=0, \mathrm{curl}\,\mathbf{u} \times n=-K\mathbf{u} \textnormal{ on } \partial \Omega \right\}$ satisfies
\be\ba\la{bkm1}
\|\na \mathbf{u}\|_{L^\infty} \le C \left( \|\div \mathbf{u} \|_{L^\infty}
+ \|\curl \mathbf{u}\|_{L^\infty} \right) \log \left(e+ \|\na^2 \mathbf{u}\|_{L^q} \right)+ C\|\na \mathbf{u}\|_{L^2}+C.
\ea\ee
\end{lemma}

Finally, we introduce the following "inversion" operator of the divergence, whose details can be found in \cite{CL}.
\begin{lemma}\la{iod}
For $1<p<\infty$, there exists a bounded linear operator
\be\ba\nonumber
\mathcal{B}:\left\{f \in L^p(\OM) : \  \int_\Omega fdx=0\right\}&\rightarrow (W^{1,p}_0(\OM))^3,
\ea\ee
such that $v=\mathcal{B}(f)$ solves
\be\la{iod1}\ba
\begin{cases}
\mathrm{div}v=f&\ \textnormal{ in }\Omega, \\
v=0&\ \textnormal{ on } \partial\Omega.
\end{cases}
\ea\ee
Furthermore, the operator satisfies the following properties:

(1) For $1<p<\infty$, there exists a constant $C$ depending on $\Omega$ and $p$ such that
\bnn
\|\mathcal{B}(f)\|_{W^{1,p}}\leq C(p)\|f\|_{L^p}.
\enn

(2) If $f=\mathrm{div} h$, for some $h \in L^p$ with $h \cdot n=0$ on $\partial\Omega$, then $v=\mathcal{B}(f)$ is a weak solution of the problem (\ref{iod1}) and satisfies
\bnn
\|\mathcal{B}(f)\|_{L^p}\leq C(p) \| h \|_{L^p}.
\enn
\end{lemma}

\section{A Priori Estimates (\uppercase\expandafter{\romannumeral1}): Lower Order Estimates}
In this section, we assume that the initial data $(\n_0,\mathbf{u}_0)$ satisfy (\ref{wsol1}) and $\n_0 >0$.
Let $(\n,\mathbf{u})$ be an axisymmetric classical solution of (\ref{ns})--(\ref{i3}) on $\OM \times (0,T]$, whose existence is guaranteed by Lemmas \ref{lct} and \ref{dcx}.
Suppose further that this solution satisfies (\ref{sdc}), (\ref{csol4}) and
\be\la{mdsj}\ba
\n(x,t) \leq 2 \left( 1+\|\n_0\|_{L^\infty} \right) e^{ \frac{\ga-1}{\ga |\OM|} E_0 }
\quad \mathrm{for\ all\ } (x,t)\in\OM \times[0,T].
\ea\ee
We define the effective viscous flux $G$ by
\be\ba\la{gw}
G \triangleq (2\mu + \lam)\div \mathbf{u} - (P-\ol{P}).
\ea\ee
Furthermore, we introduce the quantities
\be\ba\nonumber
A_1^2(t)\triangleq \int \left( \nu (\div \mathbf{u})^2 + |\na \mathbf{u}|^2
+ \frac{1}{\nu} (P-\ol{P})^2 \right) dx,
\ea\ee
and
\be\ba\nonumber
A_2^2(t)\triangleq \int \rho |\dot{\mathbf{u}}|^2 dx.
\ea\ee

We first state the standard energy estimate.
\begin{lemma}\la{l1}
There exists a positive constant
$C$ depending only on
$\ga$, $\mu$, $\|\n_0\|_{L^\infty}$, $\| \mathbf{u}_0 \|_{H^1}$, and $K$ such that
\be\ba\la{zdc1}
\sup_{0\leq t\leq T}\left( \int \frac{1}{2}\rho |\mathbf{u}|^2 + \frac{P}{\ga-1} dx \right) 
+ \int_0^T \left( \nu \| \div \mathbf{u} \|^2_{L^2} + \| \na \mathbf{u} \|^2_{L^2} \right) dt
\le C.
\ea\ee
\end{lemma}
\begin{proof}
First, from $(\ref{ns})_1$, we deduce that $P$ satisfies
\be\la{zdc22}\ba
P_t+\div( P \mathbf{u} )+(\gamma-1)P\div \mathbf{u}=0.
\ea\ee
Integrating (\ref{zdc22}) over $\OM$ and applying the boundary condition (\ref{i3}) yield
\be\la{zdc23}\ba
\frac{d}{dt} \int \frac{P}{\ga-1} dx + \int P\div \mathbf{u} dx = 0.
\ea\ee
Next, multiplying $(\ref{ns})_2$ by $\mathbf{u}$ and integrating by parts over $\OM$, with the help of (\ref{zdc23}) and the boundary condition (\ref{i3}), we obtain
\be\la{zdc11}\ba
& \frac{d}{dt} \left( \int \frac{1}{2}\rho |\mathbf{u}|^2 + \frac{P}{\ga-1} dx \right)
+ (2\mu + \lam) \int (\div \mathbf{u})^2 dx + \mu \int |\curl \mathbf{u}|^2 dx \\
& \quad + \mu \int_{\p \OM} \mathbf{u} \cdot K \cdot \mathbf{u} ds = 0,
\ea\ee
where we have used the identity: $\Delta \mathbf{u} = \na \div \mathbf{u} - \na \times \curl \mathbf{u}$.

Then, Lemma \ref{dltdc2} combined with (\ref{zdc11}) gives
\be\ba\nonumber
& \frac{d}{dt} \left( \int \frac{1}{2}\rho |\mathbf{u}|^2 + \frac{P}{\ga-1} dx \right)
+ \lam \| \div \mathbf{u} \|^2_{L^2} + 2 \mu \| D(\mathbf{u}) \|^2_{L^2} \\
& \quad + \mu \int_{\p \OM} \mathbf{u} \cdot ( K+2D(n) ) \cdot \mathbf{u} ds = 0,
\ea\ee
which together with Lemma \ref{dltdc1} shows that
\be\la{zdc14}\ba
& \frac{d}{dt} \left( \int \frac{1}{2}\rho |\mathbf{u}|^2 + \frac{P}{\ga-1} dx \right)
+ \lam \| \div \mathbf{u} \|^2_{L^2} + \frac{\mu}{\Lambda} \| \na \mathbf{u} \|^2_{L^2} \le 0.
\ea\ee
Finally, integrating (\ref{zdc14}) over $(0,T)$ implies (\ref{zdc1}) and completes the proof of Lemma \ref{l1}.
\end{proof}

\begin{lemma}\la{l2}
There exists a positive constant
$C$ depending only on
$\ga$, $\mu$, $\|\n_0\|_{L^\infty}$, $\| \mathbf{u}_0 \|_{H^1}$, and $K$ such that
\be\la{zdc02}\ba
\int_0^T \int(P-\ol{P})^{2} dxdt\le C \nu.
\ea\ee
\end{lemma}
\begin{proof}
Multiplying $(\ref{ns})_2$ by $\mathcal{B}[P-\ol{P} ]$, integrating over $\Omega$,
and applying (\ref{mdsj}), (\ref{zdc1}) and H\"older's inequality, we conclude that
\be\la{zdc24}\ba
\int(P-\ol{P} )^2 dx 
&= \left(\int\rho \mathbf{u} \cdot \mathcal{B}[P-\ol{P}] dx\right)_t
- \int\rho \mathbf{u}\cdot\mathcal{B}[P_t-\ol{P_t}]dx  \\
& \quad -\int\rho \mathbf{u} \cdot\nabla\mathcal{B}[P-\ol{P}]\cdot \mathbf{u} dx +\mu \int \p_i \mathbf{u} \cdot \p_i \mathcal{B}[P-\ol{P}] dx \\
& \quad  + (\mu+\lambda)\int(P-\ol{P})\div \mathbf{u} dx \\
& \le \left(\int\rho \mathbf{u} \cdot\mathcal{B}[P-\ol{P}] dx\right)_t+\|\n  \mathbf{u} \|_{L^2}\|\mathcal{B}[P_t-\ol{P_t} ]\|_{L^2} \\
& \quad +C \| \n \|_{L^4} \| \mathbf{u} \|_{L^4}^{2} \|P-\ol P\|_{L^4} 
+C\|P-\ol P\|_{L^2} \left( \|\nabla \mathbf{u} \|_{L^2}+ \nu \| \div \mathbf{u} \|_{L^2} \right) \\
& \leq \left(\int\rho \mathbf{u} \cdot\mathcal{B}[P-\ol{P}] dx\right)_t
+\frac{1}{2} \|P-\ol{P}\|_{L^2}^2 +C \left( \|\nabla \mathbf{u} \|^2_{L^2}
+ \nu^2 \| \div \mathbf{u} \|^2_{L^2} \right),
\ea\ee
where in the last inequality we have used the following estimate:
\bnn\ba
\|\mathcal{B}[P_t-\ol{P_t}]\|_{L^2}
&=\|\mathcal{B} [\div(P\mathbf{u})] 
+ (\ga-1) \mathcal{B} [P\div \mathbf{u} - \ol{P\div \mathbf{u}}]\|_{L^2} \\
& \le C\left( \| P \mathbf{u} \|_{L^2}+ \| P \div \mathbf{u} \|_{L^2} \right) \\
&\le C \|\na \mathbf{u}\|_{L^2},
\ea\enn
due to (\ref{zdc22}), (\ref{zdc23}) and Lemma \ref{iod}.

Integrating (\ref{zdc24}) over $(0,T)$ and using (\ref{mdsj}), (\ref{zdc1}) and Lemma \ref{iod},
we derive (\ref{zdc02}) and finish the proof of Lemma \ref{l2}.
\end{proof}

For $2\le p<\infty$, the following estimate of $\| \na \mathbf{u} \|_{L^p}$ is crucial and will be used extensively in the subsequent analysis.
\begin{lemma}\la{l3}
For any $2 \le p<\infty$, there exists a positive constant $C$ depending only on
$p$, $\ga$, $\mu$, $\|\n_0\|_{L^\infty}$, $\| \mathbf{u}_0 \|_{H^1}$, and $K$ such that
\be\la{zdc03}\ba
\| \nabla \mathbf{u} \|_{L^{p}}
\le C \left( A_1^{\frac{2}{p}} A_2^{1-\frac{2}{p}}
+ \frac{1}{\nu} \| P-\ol{P} \|_{L^p} + A_1 \right).
\ea\ee
\end{lemma}
\begin{proof}
First, choosing $\mathbf{f}=\mathbf{u}$, $\mathbf{f}=\curl \mathbf{u}$, and $g=G$ in (\ref{ewgn01}), respectively, we obtain:
\be\la{zdc38}\ba
\| \mathbf{u} \|_{L^p}
\le C \| \mathbf{u} \|^{\frac{2}{p}}_{L^2} \| \mathbf{u} \|^{1-\frac{2}{p}}_{H^1}, \quad
\| \curl \mathbf{u} \|_{L^p} \le C \| \curl \mathbf{u} \|^{\frac{2}{p}}_{L^2} \| \curl \mathbf{u} \|^{1-\frac{2}{p}}_{H^1},
\ea\ee
and
\be\la{zdc313}\ba
\| G \|_{L^p} \le C \| G \|^{\frac{2}{p}}_{L^2} \| G \|^{1-\frac{2}{p}}_{H^1}.
\ea\ee
Then, by (\ref{gw}) we rewrite $(\ref{ns})_2$ as
\be\la{zdc314}\ba
\n\dot{u}= \na G - \mu \na\times \curl \mathbf{u}.
\ea\ee
Combining this with the boundary condition (\ref{i3}), we find that $G$ satisfies the following elliptic equation:
\be\la{zdc315}\ba
\begin{cases}
\Delta G=\div \left( \rho \dot{\mathbf{u}} -\mu \na \times (K \mathbf{u})^\perp \right) & \mathrm{in}\, \,  \OM, \\
\frac {\p G}{\p n}= \left( \rho \dot{\mathbf{u}} - \mu \na \times (K \mathbf{u})^\perp \right) \cdot n &\mathrm{on}\, \,  \p \OM,
\end{cases}
\ea\ee
where $(K \mathbf{u})^\perp \triangleq - (K \mathbf{u}) \times n$.

The standard $L^p$ estimate of elliptic equations (see \cite[Lemma 4.27]{NS}) implies that for any integer $k \ge 0$ and $1<p<\infty$,
\be\la{zdc316}\ba
\| \na G \|_{W^{k,p}} \le C \left( \| \n \dot{\mathbf{u}}\|_{W^{k,p}}
+ \| \na \times (K \mathbf{u})^\perp \|_{W^{k,p}} \right),
\ea\ee
where $C$ depends only on $\mu$, $p$, $k$, and $\OM$.

Observing that $(\curl \mathbf{u} + (K \mathbf{u})^\perp) \times n = 0$ on $\p \OM$ and
$\div(\na \times \curl \mathbf{u}) = 0$, and using (\ref{zdc314}), (\ref{zdc316}), and Lemma \ref{dc}, we derive
\be\la{zdc317}\ba
\| \na \curl \mathbf{u} \|_{W^{k,p}}
\le C \left( \| \n \dot{\mathbf{u}}\|_{W^{k,p}}
+ \| \na (K \mathbf{u})^\perp \|_{W^{k,p}} + \| \na \mathbf{u} \|_{L^p} \right).
\ea\ee
In particular, from (\ref{zdc316}), (\ref{zdc317}) and Poincar\'e's inequality, we deduce that
\be\la{zdc318}\ba
\| G \|_{H^1} + \| \curl \mathbf{u} \|_{H^1}
\le C \left( \| \rho \dot{\mathbf{u}} \|_{L^2} + \| \na \mathbf{u} \|_{L^2} \right).
\ea\ee
By virtue of (\ref{dc1}), (\ref{mdsj}), (\ref{zdc38}), (\ref{zdc313}), (\ref{zdc318}), and Poincar\'e's inequality, we have
\be\ba\nonumber
\| \na \mathbf{u} \|_{L^p}
& \le C \left(\| \div \mathbf{u} \|_{L^p}
+ \| \curl \mathbf{u} \|_{L^p} + \| \mathbf{u} \|_{L^p} \right) \\
& \le C \left( \frac{1}{\nu} \| G \|_{L^p}+ \frac{1}{\nu} \| P-\ol{P} \|_{L^p}
+ \| \curl \mathbf{u} \|_{L^p}
+ \| \mathbf{u} \|^{\frac{2}{p}}_{L^2} \| \mathbf{u} \|^{1-\frac{2}{p}}_{H^1} \right) \\
& \le C \left( \frac{1}{\nu} \| G \|^{\frac{2}{p}}_{L^2}
\| G \|^{1-\frac{2}{p}}_{H^1}
+ \frac{1}{\nu} \| P-\ol{P} \|_{L^p} + \| \curl \mathbf{u} \|^{\frac{2}{p}}_{L^2}
\| \curl \mathbf{u} \|^{1-\frac{2}{p}}_{H^1} + \| \na \mathbf{u} \|_{L^2} \right) \\
& \le C A_1^{\frac{2}{p}} (A_1+A_2)^{1-\frac{2}{p}} + \frac{C}{\nu} \| P-\ol{P} \|_{L^p} + C A_1 \\
& \le C \left( A_1^{\frac{2}{p}} A_2^{1-\frac{2}{p}}
+ \frac{1}{\nu} \| P-\ol{P} \|_{L^p} + A_1 \right),
\ea\ee
which gives (\ref{zdc03}) and completes the proof of Lemma \ref{l3}.
\end{proof}

\begin{lemma}\la{l4}
There exists a positive constant $C$ depending only on
$\ga$, $\mu$, $\|\n_0\|_{L^\infty}$, $\| \mathbf{u}_0 \|_{H^1}$, and $K$ such that
\be\la{zdc04}\ba
\sup_{0 \le t \le T} \log \left( 2 + A^2_1 \right)
+ \int_0^T \frac{A^2_2}{2 + A^2_1} dt
\le C \log \left( 2 + A^2_1(0) \right).
\ea\ee
\end{lemma}
\begin{proof}
First, multiplying $(\ref{ns})_2$ by $\mathbf{u}_t$ and integrating over $\OM$, by $(\ref{ns})_1$, (\ref{i3}) and (\ref{gw}), we derive
\be\ba\nonumber
& \frac{d}{dt} \left( \frac{\nu}{2}\|\div \mathbf{u} \|_{L^2}^2 
+ \frac{\mu}{2} \| \curl \mathbf{u} \|^2_{L^2}
+ \frac{\mu}{2} \int_{\p \OM} \mathbf{u} \cdot K \cdot \mathbf{u} ds \right)
+ A^2_2 \\
&= \int (P-\ol{P}) \div \mathbf{u}_t dx + \int \n \dot{\mathbf{u}} \cdot (\mathbf{u} \cdot \na) \mathbf{u} dx \\
& = \int \n \dot{\mathbf{u}} \cdot (\mathbf{u} \cdot \na) \mathbf{u} dx
+ \frac{d}{dt} \int (P-\ol{P}) \div \mathbf{u} dx- \int P_t \div \mathbf{u} dx \\
& = \int \n \dot{\mathbf{u}} \cdot (\mathbf{u} \cdot \na) \mathbf{u} dx
+\frac{d}{dt} \int (P-\ol{P}) \div \mathbf{u} dx- \frac{1}{\nu}\int P_t G dx
-\frac{1}{\nu}\int P_t (P-\ol{P}) dx \\
& \le \frac{d}{dt} \left(\int (P-\ol{P})\div \mathbf{u} dx
-\frac{1}{2\nu}\| P-\ol{P} \|^2_{L^2} \right) 
+ \frac{1}{2} A^2_2 + \frac{1}{2} \int \n |\mathbf{u}|^2 |\na \mathbf{u}|^2 dx \\
& \quad - \frac{1}{\nu} \int P_t G dx,
\ea\ee
which implies
\be\ba\la{zdc41}
\frac{d}{dt}B_1(t) + \frac{1}{2} A^2_2
& \le - \frac{1}{\nu} \int P_t G dx + \frac{1}{2} \int \n |\mathbf{u}|^2 |\na \mathbf{u}|^2 dx,
\ea\ee
with
\be\ba\la{B1}
B_1(t) & \triangleq \frac{\nu}{2} \| \div \mathbf{u} \|_{L^2}^2 + \frac{\mu}{2}\|\curl \mathbf{u} \|^2_{L^2}
+ \frac{\mu}{2} \int_{\p \OM} \mathbf{u} \cdot K \cdot \mathbf{u} ds \\
& \quad + \frac{1}{2\nu}\| P-\ol{P} \|^2_{L^2} - \int( P-\ol{P} ) \div \mathbf{u} dx.
\ea\ee
It follows from (\ref{zdc22}), (\ref{gw}), (\ref{zdc318}), and Poincar\'e's inequality that
\be\ba\la{zdc42}
- \frac{1}{\nu} \int P_t G dx & = -\frac{1}{\nu}\int P \mathbf{u} \cdot\na G dx 
+ \frac{\ga-1}{\nu} \int P \div \mathbf{u} G dx \\
& \leq \frac{C}{\nu} \| \mathbf{u} \|_{L^2} \|\na G\|_{L^2}
+ \frac{C}{\nu} \| \div \mathbf{u} \|_{L^2} \| G \|_{L^2} \\
& \leq \frac{C}{\nu} \| \na \mathbf{u} \|_{L^2} \left( \| \n \dot{\mathbf{u}} \|_{L^2} + \| \na \mathbf{u} \|_{L^2} \right) \\
& \le \frac{1}{8} A^2_2 + C A^2_1.
\ea\ee
In addition, from (\ref{sdc}), we obtain
\be\la{zdc43}\ba
\int_\OM \n^2 |\mathbf{u}|^4 dx
\le C \int_D \n^2 \left( |u_r|^4 + |u_\theta|^4 + |u_z|^4 \right) drdz.
\ea\ee
In view of (\ref{logbd1}), (\ref{ewgn01}) and Poincar\'e's inequality, it holds that
\be\la{zdc44}\ba
\int_D r \n^2 |u_r|^4 drdz
& \le C \| \sqrt{\n} u_r \|^2_{L^2(D)}
+ C \left( 1 + \| \sqrt{\n} u_r \|^2_{L^2(D)} \right) \| u_r \|^2_{H^1(D)}
\log \left( 2+\| u_r \|^2_{H^1(D)} \right) \\
& \le C \| \sqrt{\n} \mathbf{u} \|^2_{L^2(\OM)}
+ C \left( 1 + \| \sqrt{\n} \mathbf{u} \|^2_{L^2(\OM)} \right) 
\|\mathbf{u}\|^2_{H^1(\OM)} \log \left( 2+\| \mathbf{u} \|^2_{H^1(\OM)} \right) \\
& \le C \| \na \mathbf{u} \|^2_{L^2(\OM)} \log \left( 2+\| \na \mathbf{u} \|^2_{L^2(\OM)} \right).
\ea\ee
Similarly, we also have
\be\ba\nonumber
\int_D r \n^2 \left( |u_\theta|^4 + |u_z|^4 \right) drdz
\le C \| \na \mathbf{u} \|^2_{L^2(\OM)} \log \left( 2+\| \na \mathbf{u} \|^2_{L^2(\OM)} \right),
\ea\ee
which together with (\ref{zdc43}) and (\ref{zdc44}) leads to
\be\ba\nonumber
\| \sqrt{\n} \mathbf{u} \|^4_{L^4}
\le C A^2_1 \log \left( 2 + A^2_1 \right).
\ea\ee
Combining this with (\ref{zdc03}) and H\"older's inequality yields
\be\ba\la{zdc47}
\frac{1}{2} \int \n |\mathbf{u}|^2 |\na \mathbf{u}|^2 dx
& \leq C\| \sqrt{\n} \mathbf{u} \|^2_{L^4} \| \na \mathbf{u} \|_{L^4}^2 \\
& \leq C \| \sqrt{\n} \mathbf{u} \|^2_{L^4}
\left( A_1 A_2 + \frac{1}{\nu} \| P-\ol{P} \|^2_{L^4} + A^2_1 \right) \\
& \leq \frac{1}{8} A^2_2
+ C \left( A^2_1 \| \sqrt{\n} \mathbf{u} \|^4_{L^4}
+ \| \sqrt{\n} \mathbf{u} \|^4_{L^4} + \frac{1}{\nu^4} \| P-\ol{P} \|^4_{L^4} + A^4_1 \right) \\
& \le \frac{1}{8} A^2_2
+ C A^2_1 \left( 1 + A^2_1 \right) \log \left( 2 + A^2_1 \right).
\ea\ee
Substituting (\ref{zdc42}) and (\ref{zdc47}) into (\ref{zdc41}), we obtain
\be\ba\la{zdc48}
\frac{d}{dt}B_1(t) + \frac{1}{4} A^2_2
& \le C A^2_1 \left( 1 + A^2_1 \right) \log \left( 2 + A^2_1 \right).
\ea\ee

Moreover, H\"older's and Young's inequalities ensure that there 
exists a positive constant $\check{C}$ depending only on $\ga$, $\mu$, $\|\n_0\|_{L^\infty}$, and $\| \mathbf{u}_0 \|_{H^1}$ such that
\be\ba\la{zdc49}
\left| \int(P-\ol{P})\div \mathbf{u} dx \right|
\le \frac{\lam}{4} \| \div \mathbf{u} \|_{L^2}^2 + \check{C}.
\ea\ee
We set
\be\ba\nonumber
B_2(t) \triangleq B_1(t) + \check{C}.
\ea\ee
From (\ref{zdc49}), (\ref{dltdc02}) and Lemma \ref{dltdc1}, we deduce that
\be\ba\nonumber
B_2(t) & = \frac{\lam}{2} \| \div \mathbf{u} \|_{L^2}^2
+ \mu \| D( \mathbf{u} ) \|^2_{L^2}
+ \frac{\mu}{2} \int_{\p \OM} \mathbf{u} \cdot ( K+2D(n) ) \cdot \mathbf{u} ds \\
& \quad + \frac{1}{2\nu}\| P-\ol{P} \|^2_{L^2} - \int( P-\ol{P} ) \div \mathbf{u} dx + \check{C} \\
& \ge \frac{\lam}{4} \| \div \mathbf{u} \|_{L^2}^2
+ \frac{\mu}{2 \Lambda} \| \na \mathbf{u} \|^2_{L^2}
+ \frac{1}{2\nu}\| P-\ol{P} \|^2_{L^2},
\ea\ee
which shows that
\be\la{zdc4100}\ba
\frac{1}{C}(2+B_2) \le 1+A^2_1 \le C (2+B_2).
\ea\ee
This combined with (\ref{dltdc1}), (\ref{zdc1}), and (\ref{zdc48}) yields
\be\ba\la{zdc411}
\frac{d}{dt} \left( 2 + B_2(t) \right) + \frac{1}{4} A^2_2
& \le C A^2_1 \left( 2 + B_2 \right) \log \left( 2 + B_2 \right).
\ea\ee
Dividing (\ref{zdc411}) by $2 + B_2(t)$, we arrive at
\be\ba\la{zdc412}
\frac{d}{dt} \log \left( 2 + B_2(t) \right) + \frac{A^2_2}{4(2 + B_2(t))}
& \leq C A^2_1 \log \left( 2 + B_2(t) \right).
\ea\ee
Applying Gr\"onwall's inequality to (\ref{zdc412}), we obtain after using (\ref{zdc1}) and (\ref{zdc02}) that
\be\ba\nonumber
\sup_{0 \le t \le T} \log \left( 2 + B_2(t) \right)
+ \int_0^T \frac{A^2_2}{4(2 + B_2(t))} dt
\le C \log \left( 2 + B_2(0) \right),
\ea\ee
which along with (\ref{zdc4100}) yields (\ref{zdc04}), thereby completing the proof of Lemma \ref{l4}.
\end{proof}

\begin{lemma}\la{l5}
There exists a positive constant $C$ depending only on
$\ga$, $\mu$, $\|\n_0\|_{L^\infty}$, $\| \mathbf{u}_0 \|_{H^1}$, and $K$ such that
\be\la{nsb03}\ba
\sup_{0\le t\le T}\int \n |\mathbf{u}|^{2+\de} dx \le C \nu,
\ea\ee
with
\be\ba\nonumber
\de \triangleq \nu^{-\frac{1}{2}} \de_0,
\ea\ee
where $\de_0 \le \frac{1}{2} \mu^{\frac{1}{2}}$ is a positive constant.
\end{lemma}
\begin{proof}
First, multiplying $(\ref{ns})_2$ by $(2+\de)|\mathbf{u}|^\de \mathbf{u}$ and integrating by parts over $\OM$, we derive
\be\la{nsb31}\ba
& \frac{1}{(2+\de)} \frac{d}{dt}\int \n |\mathbf{u}|^{2+\de} dx
+ \int |\mathbf{u}|^\de \left(\mu |\curl \mathbf{u}|^2
+ \nu (\div \mathbf{u})^2 \right) dx
+ \mu \int_{\p \OM} \mathbf{u} \cdot K \cdot \mathbf{u} |\mathbf{u}|^\de dS \\
& \le C \de  \int  \left(\nu |\div \mathbf{u}|+\mu |\curl \mathbf{u}| \right)  |\mathbf{u}|^\de |\na \mathbf{u}| dx
+ C \int |P-\ol{P}| |\mathbf{u}|^\de |\na \mathbf{u}|dx \\
& \triangleq J_1+J_2.
\ea\ee
From (\ref{wdc1}) and Cauchy's inequality, we deduce that
\be\la{nsb32}\ba
J_1 & \le \frac{1}{2} \int |\mathbf{u}|^\de \left(\mu |\curl \mathbf{u}|^2
+ \nu (\div \mathbf{u})^2 \right) dx
+ \frac{C\de^2 \nu}{2} \int |\mathbf{u}|^\de |\na \mathbf{u}|^2 dx \\
& \le \frac{1+C_1\de^2 \nu}{2} \int |\mathbf{u}|^\de \left(\mu |\curl \mathbf{u}|^2
+ \nu (\div \mathbf{u})^2 \right) dx + C\de^2 \nu \| \mathbf{u} \|^{2+\de}_{L^{2+\de}},
\ea\ee
provided $\de \in (0,\hat{\de})$, where $C_1$ depends only on $\mu$.

For $J_2$, Young's and Poincar\'e's inequalities give
\be\la{nsb33}\ba
J_2 &\le C \int |P-\ol{P}| \left( 1 + |\mathbf{u}| \right)  |\na \mathbf{u}| dx \\
& \le C \| P-\ol{P} \|_{L^2} \| \na \mathbf{u}\|_{L^2}
+ C \| \mathbf{u} \|_{L^2} \| \na \mathbf{u} \|_{L^2} \\
& \le C \left( \| P-\ol{P} \|^2_{L^2} + \| \na \mathbf{u} \|^2_{L^2} \right).
\ea\ee
On the other hand, Poincar\'e's inequality ensures that
\be\la{nsb34}\ba
\left| \mu \int_{\p \OM} \mathbf{u} \cdot K \cdot \mathbf{u} |\mathbf{u}|^\de dS \right|
\le C \| \mathbf{u} \|^{2+\de}_{H^1} \le C \| \na \mathbf{u} \|^{2+\de}_{L^2}.
\ea\ee

Moreover, Lemma \ref{l4} implies that there exists a positive constant $C_2$ depending only on
$\ga$, $\mu$, $\|\n_0\|_{L^\infty}$, $\| \mathbf{u}_0 \|_{H^1}$, and $K$ such that
\be\la{nsb36}\ba
\sup_{0 \le t \le T} \| \na \mathbf{u} \|^2_{L^2} \le (2+\nu)^{C_2}.
\ea\ee
Define
\be\la{de0}\ba
\de_0 \triangleq \min \left\{ \frac{1}{2} \sqrt{\mu},
\sqrt{\mu}\hat{\de},
\frac{1}{\sqrt{2C_1}},
\frac{ 2 \sqrt{\mu} }{C_2} \right\}.
\ea\ee
Substituting (\ref{nsb32}) and (\ref{nsb33}) into (\ref{nsb31})
and applying (\ref{nsb34}), (\ref{nsb36}) and (\ref{de0}), we obtain
\be\la{nsb35}\ba 
\frac{d}{dt}\int \n |\mathbf{u}|^{2+\de}dx \le C \left( \| P-\ol{P} \|^2_{L^2}
+ \nu \| \na \mathbf{u} \|^2_{L^2} \right).
\ea\ee

Finally, integrating (\ref{nsb35}) over $(0,T)$ and using
(\ref{zdc1}) and (\ref{zdc02}), we arrive at (\ref{nsb03}) and complete the proof of Lemma \ref{l5}.
\end{proof}

Next, we estimate the upper bound of the density, for which a crucial step is to obtain pointwise bounds of the effective viscous flux $G$.
To this end, we adopt the method in \cite{FLL,LKV}, which derives a pointwise estimate of $G$ through Green's function and conformal mapping.

Since the solution is assumed to be periodic in the $x_3$-direction, following \cite{LKV}, we extend $\OM$ in the $x_3$-direction to a larger domain $\OM_1$, and establish estimates for $G$ on $\OM$ by working in $\OM_1$.
For this purpose, we set
\be\la{o1}\ba
\OM_1 \triangleq \{ (x_1,x_2,x_3) \in \rrr : 1<x^2_1+x^2_2<4, -2<x_3<3 \}.
\ea\ee

From the periodicity in $x_3$ and (\ref{zdc315}), we deduce that for any $t\in [0,T]$, $G$ satisfies
the following elliptic equation with Neumann boundary conditions:
\be\la{evf1}\ba
\begin{cases}
	\Delta G=\div \left( \rho \dot{\mathbf{u}} \right) & \mathrm{in}\, \,  \OM_1, \\
	\frac {\p G}{\p n}= \left( \rho \dot{\mathbf{u}} - \mu \na \times (K \mathbf{u})^\perp \right) \cdot n &\mathrm{on}\, \,  \p \OM_1.
\end{cases}
\ea\ee

Exploiting the axisymmetry of the problem, we rewrite the above equation in two-dimensional form.
Let $\tilde{\Delta} \triangleq \p_{rr}+ \p_{zz}$ and $\tilde{\na} \triangleq (\p_{r}, \p_{z})$,
a straightforward computation yields
\be\la{evf2}\ba
\begin{cases}
	\tilde{\Delta} G=\div \left( \rho \dot{\mathbf{u}} \right) - \frac{1}{r}\p_r G & \mathrm{in}\, \,  D_1, \\
	\tilde{\na} G \cdot \tilde{n} = \left( \rho \dot{\mathbf{u}} - \mu \na \times (K \mathbf{u})^\perp \right) \cdot n &\mathrm{on}\, \,  \p D_1,
\end{cases}
\ea\ee
where $D_1 \triangleq \{(r,z) \in \rr :1<r<2, -2<z<3 \}$ and $\tilde{n}$ denotes the unit outer normal vector of the boundary $\p D_1$.

The Green's function $N(x,y)$ for the Neumann problem on the unit disc $\mathbb{D}$ (see \cite{STT}) is expressed as
\be\ba\nonumber
N(x,y)=-\frac{1}{2\pi}\bigg(\log|x-y|+\log\left||x|y-\frac{x}{|x|}\right|\bigg).
\ea\ee

Moreover, according to the Riemann mapping theorem (see \cite{SES}), there exists a conformal mapping
$\varphi=(\varphi_1, \varphi_2):\overline{D_1}\rightarrow\overline{\mathbb{D}}$.
We define the pull back Green's function $\widetilde{N}(x,y)$ of $D_1$ by
\be\ba\nonumber
\widetilde{N}(x,\, y)=N\big(\varphi(x),\varphi(y)\big) \ \ \mathrm{for}\ x,y\in D_1.
\ea\ee

For any $\mathbf{x}=(x_1,x_2,x_3) \in \OM_1$,
we denote the corresponding two-dimensional coordinates by $x=(r_\mathbf{x},z_\mathbf{x}) \in D_1$ under the axisymmetric coordinate transformation,
where $r_\mathbf{x}=\sqrt{x^2_1+x^2_2}$ and $z_\mathbf{x}=x_3$.
We also set $u_1 \triangleq u_r$ and $u_2 \triangleq u_z$.

Building upon the preceding notation and following the arguments in \cite[Lemmas 3.7 \& 3.8]{LKV}, we arrive at the following estimate for $G$.

\begin{lemma}\label{l6}
Assume that $G\in C\big([0,T];C^1(\overline{\OM_1})\cap C^2(\OM_1)\big)$ satisfies the equation $(\ref{evf1})$.
Then for any $\mathbf{x}\in \OM$, there exists a positive constant $C$ depending only on
$\ga$, $\mu$, $\|\n_0\|_{L^\infty}$, $\| \mathbf{u}_0 \|_{H^1}$, and $K$ such that
\be\la{zdcg01}\ba
-G(\mathbf{x},t)
& \le \frac{D}{Dt} \psi (\mathbf{x},t) 
+ C \left( \| \sqrt{\n} \dot{\mathbf{u}} \|_{L^2}
+ \| \na \mathbf{u} \|^2_{L^2} + \| G \|_{H^1}
+ \| \na \mathbf{u} \|_{L^4} \right) - J,
\ea\ee
where
\be\la{zdcg02}\ba
\psi \triangleq \int_{D_1}  \left( \p_{y_i} \widetilde{N}(x,y) \n u_i(y) \right) dy,
\ea\ee
and $J$ satisfies
\be\la{zdcg03}\ba
|J| & \le C \| \na \mathbf{u} \|^2_{L^2}
+ C \sup_{x \in \ol{D_1}} \left( \sum^{2}_{i,j=1} \int_{D_1} \frac{\left| u_i(x) - u_i(y) \right|}{|x-y|^2} \n |u_j| (y) dy \right).
\ea\ee
\end{lemma}

\begin{lemma}\la{l7}
There exists a positive constant $\nu_1$ depending only on 
$\ga$, $\mu$, $\|\n_0\|_{L^\infty}$, $\| \mathbf{u}_0 \|_{H^1}$, and $K$
such that, if $(\n,\mathbf{u})$ satisfies that
\be\ba\la{zdc07}
\sup_{0\leq t\leq T} \|\n\|_{L^\infty} 
\leq 2 \left( 1+\|\n_0\|_{L^\infty} \right) e^{ \frac{\ga-1}{\ga |\OM|} E_0 },
\ea\ee
then 
\be\ba\la{zdc007}
\sup_{0\leq t\leq T} \|\n\|_{L^\infty} 
\leq \frac{3}{2} \left( 1+\|\n_0\|_{L^\infty} \right) e^{ \frac{\ga-1}{\ga |\OM|} E_0 },
\ea\ee
provided $\nu \geq \nu_1$.
\end{lemma}
\begin{proof}
First, note that the mass equation $(\ref{ns})_1$ and the condition $\n_0>0$ imply $\n>0$.
We rewrite $(\ref{ns})_1$ using (\ref{gw}) as
\be\ba\nonumber
\pa_t \log\n + \mathbf{u} \cdot \na \log\n + \frac{1}{\nu}( P-\ol{P} + G) = 0,
\ea\ee
which combined with (\ref{zdcg01}) yields
\be\ba\la{zdc72}
& \pa_t F + \mathbf{u} \cdot \na F + \frac{1}{\nu}(P-\ol{P}) \\
& \le \frac{C}{\nu} \left( \| \sqrt{\n} \dot{\mathbf{u}} \|_{L^2}
+ \| \na \mathbf{u} \|^2_{L^2} + \| G \|_{H^1}
+ \| \na \mathbf{u} \|_{L^4} \right) - \frac{1}{\nu} J,
\ea\ee
with
\be\ba\la{F}
F \triangleq \log\n - \frac{1}{\nu}\psi.
\ea\ee
It follows from (\ref{zdc03}) and (\ref{zdc318}) that
\be\la{zdc73}\ba
& \frac{1}{\nu} \left( \| G \|_{H^1} + \| \na \mathbf{u} \|_{L^4} \right) \\
& \le \frac{C}{\nu} \left( A_1 + A_2 + A_1^{\frac{1}{2}} A_2^{\frac{1}{2}}
+ \frac{1}{\nu} \| P-\ol{P} \|_{L^4} \right) \\
& \le C \nu^{-\frac{3}{2}} + C \nu^{-\frac{1}{2}} A^2_1 + \frac{C}{\nu} A_2.
\ea\ee

Now, we estimate $J$.
For any $x,y \in \overline{D_1}$,
the Sobolev embedding theorem (Theorem 4 of \cite[Chapter 5]{EL}) shows that
\be\la{zdcg360}\ba
& |u_r(x) - u_r(y)| + |u_z(x) - u_z(y)| \\
& \le C \left( \| \tilde{\na} u_r \|_{L^4(D_1)} + \| \tilde{\na} u_z \|_{L^4(D_1)} \right) |x-y|^{\frac{1}{2}} \\
& \le C \| \na \mathbf{u} \|_{L^4(\OM)} |x-y|^{\frac{1}{2}}.
\ea\ee
Similar to the arguments in \cite[Lemma 3.8]{LX3}, we derive that
\be\la{zdcg37}\ba
\int_{D_1} |x-y|^{-\frac{3}{2}} \n ( |u_r| +|u_z|) dy
\le C \nu^{\frac{1}{2}} \| \na \mathbf{u} \|^{\frac{1}{2}}_{L^2}.
\ea\ee
Combining (\ref{zdcg360}) and (\ref{zdcg37}) leads to
\be\la{zdcg38}\ba
\sum^{2}_{i,j=1} \int_{D_1} \frac{\left| u_i(x) - u_i(y) \right|}{|x-y|^2} \n |u_j| (y) dy
& \le C \| \na \mathbf{u} \|_{L^4} \int_{D_1} |x-y|^{-\frac{3}{2}} \n ( |u_r| +|u_z|) dy \\
& \le C \nu^{\frac{1}{2}} \| \na \mathbf{u} \|^{\frac{1}{2}}_{L^2} \| \na \mathbf{u} \|_{L^4}.
\ea\ee
which together with (\ref{zdcg03}), (\ref{zdc03}) and Young's inequality, we arrive at
\be\la{zdc74}\ba
|J| & \le C \| \na \mathbf{u} \|^2_{L^2} + C \nu^{\frac{1}{2}} \| \na \mathbf{u} \|^{\frac{1}{2}}_{L^2}
\| \na \mathbf{u} \|_{L^4} \\
& \le C A^2_1 + C \nu^{\frac{1}{2}} A_1^{\frac{1}{2}}
\left( A_1^{\frac{1}{2}} A_2^{\frac{1}{2}}
+ \frac{1}{\nu} \| P-\ol{P} \|_{L^4} + A_1 \right) \\
& \le C \nu^{-\frac{1}{4}} + C \nu^{\frac{3}{4}} A^2_1 + C \nu^{\frac{1}{4}} A_2.
\ea\ee
Substituting (\ref{zdc73}) and (\ref{zdc74}) into (\ref{zdc72}) gives
\be\ba\nonumber
\pa_t F + \mathbf{u} \cdot \na F + \frac{1}{\nu}(P-\ol{P})
& \le C \nu^{-\frac{5}{4}} + C \nu^{-\frac{1}{4}} A^2_1 + C \nu^{-\frac{3}{4}} A_2.
\ea\ee
Moreover, observe that $\n^\ga \geq \ga\log\n + 1$, which implies that
\be\ba\la{zdc75}
\pa_t F + \mathbf{u} \cdot \na F + \frac{\ga}{\nu} F
\leq \frac{\ga}{\nu^2}|\psi|
+ C \nu^{-\frac{5}{4}} + C \nu^{-\frac{1}{4}} A^2_1 + C \nu^{-\frac{3}{4}} A_2
+ \frac{1}{\nu} \ol{P}.
\ea\ee
To handle the material derivative $\frac{D}{Dt}F =\pa_t F + \mathbf{u} \cdot \na F$,
we introduce the characteristic curve $y(s;x,t)$ defined by
\be\ba\nonumber
\begin{cases}
\frac{d}{ds}y(s) = \mathbf{u}(y,s),\\
y(t;x,t) = x.
\end{cases}
\ea\ee
This combined with (\ref{zdc75}) shows that for all $(x,t)\in\OM \times (0,T]$,
\be\ba\la{zdc705}
\frac{d}{ds}F(s) + \frac{\ga}{\nu} F(s)
& \le \frac{\ga}{\nu^2} \|\psi\|_{L^\infty}
+ C \nu^{-\frac{5}{4}} + C \nu^{-\frac{1}{4}} A^2_1 + C \nu^{-\frac{3}{4}} A_2
+ \frac{\ga-1}{\nu |\OM|} E_0,
\ea\ee
where we denote somewhat abusively $F(y(s;x,t),s)$ by $F(s)$.
Applying the maximum principle to (\ref{zdc705}), we obtain
\be\ba\la{zdc76}
F(t) \le &
e^{-\frac{\ga}{\nu}t} F(0) + \frac{\ga}{\nu^2}\int_0^t e^{-\frac{\ga}{\nu}(t-s)}
\| \psi(\cdot ,s) \|_{L^{\infty}}ds
+ C \nu^{-\frac{1}{4}} \int_0^t e^{-\frac{\ga}{\nu}(t-s)} A^2_1 ds \\
& + C \nu^{-\frac{5}{4}} \int_0^t e^{-\frac{\ga}{\nu}(t-s)} ds
+ C \nu^{-\frac{3}{4}} \int_0^t e^{-\frac{\ga}{\nu}(t-s)} A_2 ds
+ \frac{\ga-1}{\ga |\OM|}(1-e^{-\frac{\ga}{\nu}t}) E_0 \\
\le & e^{-\frac{\ga}{\nu}t}F(0) + C \nu^{-\frac{1}{6}} + \frac{\ga-1}{\ga |\OM|} E_0,
\ea\ee
where in the second inequality we have used the following estimates:
\be\ba\la{zdc77}
|\psi| & \le C \int_{D_1} |x-y|^{-1} \n (|u_r| + |u_z|) dy \\
& \le C \left( \int_{D_1} |x-y|^{-\frac{2+\de}{1+\de}} dy \right)^{\frac{1+\de}{2+\de}}
\left(\int_{D_1} \n (|u_r| + |u_z|)^{2+\de} dy \right)^{\frac{1}{2+\de}} \\
& \le C \de^{-\frac{1+\de}{2+\de}}
\left(\int_\OM \n |\mathbf{u}|^{2+\de} d\mathbf{y} \right)^{\frac{1}{2+\de}} \\
& \le C \de^{-\frac{1+\de}{2+\de}} \nu^{\frac{1}{2+\de}} \\
& \le C \nu^{\frac{3}{4}},
\ea\ee
and
\be\ba\nonumber
\int_0^t e^{-\frac{\ga}{\nu}(t-s)} A_2 ds
& = \int_0^t e^{-\frac{\ga}{\nu}(t-s)}
\frac{ A_2 }{(2+A^2_1)^{\frac{1}{2}}} (2+A^2_1)^{\frac{1}{2}} ds \\
& \le \left(\int_0^t e^{-\frac{2 \ga}{\nu}(t-s)} (2 + A^2_1) ds \right)^{\frac{1}{2}}
\left(\int_0^t \frac{A_2^2}{2+A^2_1} ds \right)^{\frac{1}{2}} \\
& \le C \nu^{\frac{1}{2}} \log^{\frac{1}{2}}(2+\nu) \\
& \le C \nu^{\frac{7}{12}},
\ea\ee
due to (\ref{zdc1}), (\ref{zdc02}), (\ref{zdc04}), (\ref{nsb03}), and H\"older's inequality.

From (\ref{zdc76}), (\ref{zdc77}), and (\ref{F}), one deduces that
\be\ba\la{zdc780}
\log \n &= F+\frac{1}{\nu} \psi \\
& \le e^{-\frac{\ga}{\nu}t} \log \|\n_0\|_{L^\infty} + \frac{1}{\nu} \left( \| \psi_0 \|_{L^\infty}+\| \psi \|_{L^\infty} \right)
+ C \nu^{-\frac{1}{6}} + \frac{\ga-1}{\ga |\OM|} E_0 \\
& \le \log \left( 1+\|\n_0\|_{L^\infty} \right) + \tilde{C} \nu^{-\frac{1}{6}}
+\frac{\ga-1}{\ga |\OM|} E_0,
\ea\ee
where the constant $\tilde{C}$ depends only on
$\ga$, $\mu$, $\|\n_0\|_{L^\infty}$, $\| \mathbf{u}_0 \|_{H^1}$, and $K$,
but is independent of $T$ and $\nu$.
Finally, we choose
\be\ba\la{zdc78}
\nu_1 = \max \left\{ \left( \frac{ \tilde{C} }{\log \frac{3}{2}} \right)^6,8\mu \right\}.
\ea\ee
Then when $\nu\geq\nu_1$, (\ref{zdc007}) holds, thereby completing the proof of Lemma \ref{l7}.
\end{proof}

\section{A Priori Estimates (\uppercase\expandafter{\romannumeral2}): Higher Order Estimates}

In this section, we fix the viscosity coefficient $\nu \ge \nu_1$, where $\nu_1$
is given by (\ref{zdc78}).
Let $(\n, \mathbf{u})$ be an axisymmetric classical solution of (\ref{ns})--(\ref{i3}) satisfying (\ref{sdc}), (\ref{csol4}), and (\ref{zdc07}).
We derive the higher-order estimates and establish exponential decay
by adapting methods analogous to those in \cite{LX4,CL,H1,FLL}.

\begin{lemma}\la{zdcgl1}
There exists a positive constant $C$ depending only on
$\ga$, $\mu$, $\|\n_0\|_{L^\infty}$, $\| \mathbf{u}_0 \|_{H^1}$, and $K$ such that
\be\ba\la{gb01}
\sup_{0\le t\le T}
\si \int\n| \dot{\mathbf{u}} |^2dx
+\int_0^{T} \si \| \na \dot{\mathbf{u}} \|^2_{L^2} dt \le C\nu^C.
\ea\ee
Moreover, if the initial data $(\n_0,\mathbf{u}_0)$ satisfy
$\div \mathbf{u}_0=0$, then we have
\be\ba\la{gb002}
\sup_{0\le t\le T}
\si \int \n| \dot{\mathbf{u}} |^2dx
+\int_0^{T} \si \| \na \dot{\mathbf{u}} \|^2_{L^2} dt \le C.
\ea\ee
\end{lemma}
\begin{proof}
The idea of this proof comes from \cite{CL,H1,FLL}.
Operating $ \dot{\mathbf{u}}^j[\frac{\pa}{\pa t}+\div(\mathbf{u}\cdot)]$ to
$(\ref{zdc314})^j,$ summing with respect to $j$,
and integrating by parts over $\OM$, one derives
\be\la{gb11}\ba
\frac{d}{dt}\left(\frac{1}{2}\int\rho|\dot{\mathbf{u}}|^2dx \right)
&=\int \bigg( {\dot{\mathbf{u}}}\cdot \nabla G_t + {\dot{\mathbf{u}}}^j\mathrm {div}(\mathbf{u} \partial _jG) \bigg) dx\\
&\quad - \mu \int \bigg( {\dot{\mathbf{u}}} \cdot \na \times \curl \mathbf{u}_t
+ {\dot{\mathbf{u}}}^j\partial _k(\mathbf{u}^k ( \na \times \curl \mathbf{u} )^j ) \bigg) dx \\
&=I_1+I_2.
\ea\ee
Integrating by parts and applying Young's inequality lead to
\be\la{gb12}\ba
I_1 & = \int_{\partial \Omega} G_t ( \dot{\mathbf{u}} \cdot n) ds
- \int \div \dot{\mathbf{u}} \left( \dot{G} - \mathbf{u} \cdot \na G \right) dx
+ \int \mathbf{u} \cdot \na \dot{\mathbf{u}}^j \p_j G dx \\
& \le \int_{\partial \Omega} G_t ( \dot{\mathbf{u}} \cdot n) ds
- \int \div \dot{\mathbf{u}} \dot{G} dx + C \| \na \dot{\mathbf{u}} \|_{L^2} \| \mathbf{u} \|_{L^6} \| \na G \|_{L^3} \\
& \le  \int_{\partial \Omega} G_t ( \dot{\mathbf{u}} \cdot n) ds
- \int \div \dot{\mathbf{u}} \dot{G} dx
+ \varepsilon \Vert \nabla \dot{\mathbf{u}} \Vert_{L^2}^2
+ C(\ep) A^2_2 \left(1 + A^4_1 \right) \\
& \quad + C(\ep) \left( A^2_1 + A^6_1 \right),
\ea\ee
where in the last inequality we have used the following estimate:
\be\la{gb14}\ba
& \| \na G \|_{L^3}+\| \na \curl \mathbf{u} \|_{L^3} \\
& \le \| \na G \|^{\frac{1}{2}}_{L^2} \| \na G \|^{\frac{1}{2}}_{L^6}
+\| \na \curl \mathbf{u} \|^{\frac{1}{2}}_{L^2} \| \na \curl \mathbf{u} \|^{\frac{1}{2}}_{L^6} \\
& \le C\left(\| \sqrt{\n} \dot{\mathbf{u}}\|_{L^2} + \| \na \mathbf{u} \|_{L^2} \right)^{\frac{1}{2}} 
\left(\| \n \dot{\mathbf{u}}\|_{L^6} + \| \na \mathbf{u} \|_{L^6} \right)^{\frac{1}{2}} \\
& \le C\left(\| \sqrt{\n} \dot{\mathbf{u}}\|_{L^2} + \| \na \mathbf{u} \|_{L^2} \right)^{\frac{1}{2}} 
\left( \| \sqrt{\n} \dot{\mathbf{u}}\|_{L^2} + \| \na \dot{\mathbf{u}}\|_{L^2} \right)^{\frac{1}{2}} \\
& \quad +C\left(\| \sqrt{\n} \dot{\mathbf{u}}\|_{L^2} + \| \na \mathbf{u} \|_{L^2} \right)^{\frac{1}{2}} 
\left( 1+\| \sqrt{\n} \dot{\mathbf{u}}\|_{L^2} + \| \na \mathbf{u} \|_{L^2} \right)^{\frac{1}{2}} \\
& \le C\left( A_2 + A_1 \right)^{\frac{1}{2}} 
\left(1 + A_1 + A_2 + \| \na \dot{\mathbf{u}} \|_{L^2} \right)^{\frac{1}{2}},
\ea\ee
due to (\ref{pt1}), (\ref{zdc03}), (\ref{zdc316}), (\ref{zdc317}), and H\"older's inequality.

For the boundary term in (\ref{gb12}), noticing that (\ref{i3}) and (\ref{bjds}) give
\be\la{gb15}\ba
&\int _{\partial \Omega }G_t({\dot{\mathbf{u}}}\cdot n) ds\\ 
&= - \int_{\partial \Omega }G_t( \mathbf{u} \cdot \na n \cdot \mathbf{u} ) ds \\
&= -\frac{d}{dt} \int _{\partial \Omega } G ( \mathbf{u} \cdot \na n \cdot \mathbf{u} )ds
+\int _{\partial \Omega } G( \mathbf{u} \cdot \na n \cdot \mathbf{u} )_t ds\\ 
&= -\frac{d}{dt} \int _{\partial \Omega } G( \mathbf{u} \cdot \na n \cdot \mathbf{u} )ds
+ \int _{\partial \Omega } G({\dot{\mathbf{u}}} \cdot \na n \cdot \mathbf{u} )
+ G( \mathbf{u} \cdot \nabla n \cdot {\dot{\mathbf{u}}}) ds  \\
&\quad - \int _{\partial \Omega } G \left( ( \mathbf{u} \cdot \nabla \mathbf{u} )\cdot \nabla n \cdot \mathbf{u} \right) ds
- \int _{\partial \Omega} G \left( \mathbf{u} \cdot \nabla n \cdot ( \mathbf{u} \cdot \nabla \mathbf{u} ) \right) ds \\
&=-\frac{d}{dt} \int_{\partial \Omega } G ( \mathbf{u} \cdot \nabla n \cdot \mathbf{u} )ds + J_1+J_2+J_3.
\ea\ee
From (\ref{pt1}), (\ref{zdc318}), and Poincar\'e's inequality, we deduce that
\be\la{gb16}\ba
J_1 & = \int _{\partial \Omega } G({\dot{\mathbf{u}}} \cdot \na n \cdot \mathbf{u} )
+ G( \mathbf{u} \cdot \nabla n \cdot {\dot{\mathbf{u}}}) ds \\
& \leq C \Vert G \Vert_{H^1} \Vert \dot{\mathbf{u}} \Vert_{H^1} \Vert \mathbf{u} \Vert_{H^1} \\
& \leq C ( \Vert \sqrt{\rho} \dot{\mathbf{u}} \Vert _{L^2} + \Vert \nabla \mathbf{u} \Vert _{L^2} )
\left( \Vert \sqrt{\n} \dot{\mathbf{u}} \Vert _{L^2} + \Vert \nabla \dot{\mathbf{u}} \Vert _{L^2} \right)
\Vert \nabla \mathbf{u} \Vert _{L^2} \\
&\leq \varepsilon \Vert \nabla \dot{\mathbf{u}} \Vert _{L^2}^2
+ C(\varepsilon ) A_2^2
+ C(\varepsilon ) A^2_1 A^2_2 + C(\varepsilon ) A^4_1.
\ea\ee
Combining (\ref{bjds}), (\ref{zdc38}), (\ref{zdc318}), and H\"older's inequality leads to
\be\la{gb17}\ba
\vert J_2\vert= & \left| - \int _{\partial \Omega } G \left( ( \mathbf{u} \cdot \nabla \mathbf{u} )\cdot \nabla n \cdot \mathbf{u} \right) ds \right| \\
= & \left| \int _{\partial \Omega} \mathbf{u}^\bot \times n \cdot \nabla \mathbf{u}^i \partial_i n_j \mathbf{u}^j G ds \right| \\
= & \left| \int _{\partial \Omega} n \cdot ( \na \mathbf{u}^i \times \mathbf{u}^\bot ) \partial_i n_j \mathbf{u}^j G ds \right| \\
= & \left| \int \div \left( ( \na \mathbf{u}^i \times \mathbf{u}^\bot ) \partial_i n_j \mathbf{u}^j G \right) dx \right| \\
= & \left| \int  \na (\partial_i n_j \mathbf{u}^j G ) \cdot ( \na \mathbf{u}^i \times \mathbf{u}^\bot )
- ( \na \mathbf{u}^i \cdot \na \times \mathbf{u}^\bot ) \partial_i n_j \mathbf{u}^j G dx \right| \\
\le & C \int \vert \nabla \mathbf{u} \vert \left( \vert G \vert \vert \mathbf{u} \vert^2
+ \vert G \vert \vert \mathbf{u} \vert \vert \nabla \mathbf{u} \vert
+ \vert \mathbf{u} \vert^2 \vert \nabla G \vert \right) dx \\
\le & C\Vert \nabla \mathbf{u} \Vert _{L^4} \left( \Vert G \Vert_{L^4} \Vert \mathbf{u} \Vert _{L^4}^2
+ \Vert G \Vert _{L^4} \Vert \mathbf{u} \Vert_{L^4} \Vert \nabla \mathbf{u} \Vert _{L^4}
+ \Vert \nabla G \Vert _{L^2} \Vert \mathbf{u} \Vert_{L^8}^2\right) \\
\le & C \left(\Vert \nabla \mathbf{u} \Vert _{L^4} \Vert \nabla \mathbf{u} \Vert^2_{L^2}
+ \Vert \nabla \mathbf{u} \Vert^2_{L^4} \Vert \nabla \mathbf{u} \Vert_{L^2} \right) \Vert G \Vert _{H^1} \\
\le & C \Vert \nabla \mathbf{u} \Vert^2_{L^4} \Vert \nabla \mathbf{u} \Vert_{L^2}
\left( \Vert \sqrt{\rho} {\dot{\mathbf{u}}} \Vert_{L^2} + \| \na \mathbf{u} \|_{L^2} \right) \\
\le & C A^2_1 A^2_2 + C \Vert \nabla \mathbf{u} \Vert^4_{L^4},
\ea\ee
where we have used the following fact:
\be\ba\la{gb18}
\div ( \na \mathbf{u}^i \times \mathbf{u}^\bot ) = -\na \mathbf{u}^i \cdot \na \times \mathbf{u}^\bot.
\ea\ee
Similarly, we have
\be\ba\nonumber
\vert J_3 \vert \le C A^2_1 A^2_2 + C \Vert \nabla \mathbf{u} \Vert^4_{L^4},
\ea\ee
which together with (\ref{gb15}), (\ref{gb16}) and (\ref{gb17}) implies
\be\ba\nonumber
\int _{\partial \Omega} G_t ({\dot{\mathbf{u}}}\cdot n) ds 
\leq & -\frac{d}{dt} \int_{\partial \Omega } G ( \mathbf{u} \cdot \nabla n \cdot \mathbf{u} )ds
+\ep \Vert \nabla {\dot{\mathbf{u}}}\Vert _{L^2}^2 \\
& + C(\ep) A^2_2 \left(1 + A^2_1 \right) + C(\varepsilon ) A^4_1
+ C(\ep) \| \na \mathbf{u} \|^4_{L^4}.
\ea\ee
For the second term on the last line of (\ref{gb12}), using (\ref{gw}), (\ref{zdc22})
and (\ref{zdc23}), we obtain
\be\ba\la{gb108}
\div \dot{\mathbf{u}} & = (\div \mathbf{u})_t + \p_i \mathbf{u}^j \p_j \mathbf{u}^i + \mathbf{u} \cdot \na \div \mathbf{u} \\
& = \frac{1}{\nu} \dot{G} + \frac{1}{\nu} (P_t + \mathbf{u} \cdot \na P) - \frac{1}{\nu} \ol{P_t}
+ \p_i \mathbf{u}^j \p_j \mathbf{u}^i \\
& = \frac{1}{\nu} \dot{G} - \frac{\ga}{\nu} P \div \mathbf{u}
+ \frac{\ga-1}{\nu |\OM|} \int P \div \mathbf{u} dx + \p_i \mathbf{u}^j \p_j \mathbf{u}^i,
\ea\ee
which along with Young's inequality gives
\be\la{gb109}\ba
- \int \div \dot{\mathbf{u}} \dot{G} dx
& = -\frac{1}{\nu} \| \dot{G} \|^2_{L^2} 
- \int \dot{G} \p_i \mathbf{u}^j \p_j \mathbf{u}^i dx \\
& \quad + \frac{\ga}{\nu} \int \dot{G} P \div \mathbf{u} dx
-\frac{\ga-1}{\nu |\OM|} \int P \div \mathbf{u} dx \int \dot{G} dx \\
& \le -\frac{1}{2\nu} \| \dot{G} \|^2_{L^2}+\frac{C}{\nu} \| \div \mathbf{u} \|^2_{L^2}
- \int \dot{G} \p_i \mathbf{u}^j \p_j \mathbf{u}^i dx.
\ea\ee
In addition, integrating by parts and applying H\"older's inequality, we derive
\be\la{gb1009}\ba
-\int \dot{G} \p_i \mathbf{u}^j \p_j \mathbf{u}^i dx
& = -\int ( G_t + \mathbf{u} \cdot \na G ) \p_i \mathbf{u} \cdot \na \mathbf{u}^i dx \\
& = -\frac{d}{dt} \left( \int G \p_i \mathbf{u} \cdot \na \mathbf{u}^i dx \right) 
+ 2 \int G \p_i \mathbf{u} \cdot \na \mathbf{u}^i_t dx \\
& \quad + \int G \div \mathbf{u} \p_i \mathbf{u} \cdot \na \mathbf{u}^i dx 
+ 2 \int G \mathbf{u} \cdot \na \p_i \mathbf{u} \cdot \na \mathbf{u}^i dx \\
& = -\frac{d}{dt} \left( \int G \p_i \mathbf{u} \cdot \na \mathbf{u}^i dx \right)
+ 2 \int G \p_i \mathbf{u} \cdot \na \dot{\mathbf{u}}^i dx \\
& \quad - 2 \int G \p_i \mathbf{u} \cdot \na \mathbf{u} \cdot \na \mathbf{u}^i dx
+ \int G \div \mathbf{u} \p_i \mathbf{u} \cdot \na \mathbf{u}^i dx \\
& \le -\frac{d}{dt} \left( \int G \p_i \mathbf{u} \cdot \na \mathbf{u}^i dx \right)
+ C \| \na \dot{\mathbf{u}} \|_{L^2} \| \na \mathbf{u} \|_{L^4} \| G \|_{L^4}
+ C \| \na \mathbf{u} \|^3_{L^4} \| G \|_{L^4} \\
& \le -\frac{d}{dt} \left( \int G \p_i \mathbf{u} \cdot \na \mathbf{u}^i dx \right)
+ \ep \| \na \dot{\mathbf{u}} \|^2_{L^2} + C(\ep) \| G \|^4_{L^4} + C(\ep) \| \na \mathbf{u} \|^4_{L^4}.
\ea\ee
Combining (\ref{zdc03}), (\ref{gb12}), (\ref{gb109}), and (\ref{gb1009}) yields
\be\la{gb110}\ba
I_1 & \le -\frac{d}{dt} \left( \int G \p_i \mathbf{u} \cdot \na \mathbf{u}^i dx
+ \int_{\partial \Omega } G ( \mathbf{u} \cdot \nabla n \cdot \mathbf{u} ) ds \right)
-\frac{1}{2\nu} \| \dot{G} \|^2_{L^2} +3\ep \| \na \dot{\mathbf{u}} \|^2_{L^2} \\
& \quad + C(\ep) \| G \|^4_{L^4}
+ C(\ep) A^2_2 \left(1 + A^4_1 \right) + C(\ep) \left( A^2_1 + A^6_1 \right),
\ea\ee
where we have used the following estimate:
\be\la{gb1100}\ba
\| \na \mathbf{u} \|^4_{L^4}
& \le C \left( A^2_1 A^2_2 + \frac{1}{\nu^4} \| P-\ol{P} \|^4_{L^4} + A^4_1 \right) \\
& \le C \left( A^2_1 A^2_2 + A^2_1 + A^4_1 \right),
\ea\ee
due to (\ref{zdc03}).

For $I_2$, integration by parts leads to
\be\la{gb111}\ba
I_2 & = - \mu \int \bigg( {\dot{\mathbf{u}}} \cdot \na \times \curl \mathbf{u}_t
+ {\dot{\mathbf{u}}}^j\partial _k(\mathbf{u}^k ( \na \times \curl \mathbf{u} )^j ) \bigg) dx \\
& = \mu \int_{\p \OM} \curl \mathbf{u}_t \times n \cdot \dot{\mathbf{u}} ds
- \mu \int \curl \dot{\mathbf{u}} \cdot \curl \mathbf{u}_t dx
+ \mu \int \mathbf{u} \cdot \na \dot{\mathbf{u}} \cdot (\na \times \curl \mathbf{u}) dx \\
& = \mu \int_{\p \OM} \curl \mathbf{u}_t \times n \cdot \dot{\mathbf{u}} ds
- \mu \int |\curl \dot{\mathbf{u}}|^2 dx + \mu \int \mathbf{u} \cdot \na \curl \mathbf{u} \cdot \curl \dot{\mathbf{u}} dx \\
& \quad + \mu \int \curl \dot{\mathbf{u}} \cdot (\na \mathbf{u}^i \times \p_i \mathbf{u}) dx
+ \mu \int \mathbf{u} \cdot \na \dot{\mathbf{u}} \cdot (\na \times \curl \mathbf{u}) dx \\
& \le \mu \int_{\p \OM} \curl \mathbf{u}_t \times n \cdot \dot{\mathbf{u}} ds
- \mu \int |\curl \dot{\mathbf{u}}|^2 dx + C \| \na \dot{\mathbf{u}} \|_{L^2} \| \na \mathbf{u} \|^2_{L^4} \\
& \quad + C \| \na \dot{\mathbf{u}} \|_{L^2} \| \na \curl \mathbf{u} \|_{L^3} \| \mathbf{u} \|_{L^6} \\
& \le \mu \int_{\p \OM} \curl \mathbf{u}_t \times n \cdot \dot{\mathbf{u}} ds
- \mu \| \curl \dot{\mathbf{u}} \|^2_{L^2} + \ep \Vert \nabla \dot{\mathbf{u}} \Vert_{L^2}^2
+ C(\ep)\left(A^2_1 + A^6_1 \right) \\
& \quad + C(\ep) A^2_2 \left(1 + A^4_1 \right)
+ C(\ep) \| \na \mathbf{u} \|^4_{L^4},
\ea\ee
where we have used (\ref{gb14}) and the following fact:
\be\ba\nonumber
\curl (\mathbf{u} \cdot \na \mathbf{u}) = \mathbf{u} \cdot \na \curl \mathbf{u} + \na \mathbf{u}^i \times \p_i \mathbf{u}.
\ea\ee

We now turn to the boundary term in (\ref{gb111}).
By applying (\ref{i3}), (\ref{bjds}), (\ref{pt1}), (\ref{gb18}) and Young's inequality, one obtains
\be\ba\nonumber
\mu \int_{\p \OM} \curl \mathbf{u}_t \times n \cdot \dot{\mathbf{u}} ds
& = - \mu \int_{\p \OM} \mathbf{u}_t \cdot K \cdot \dot{\mathbf{u}} ds \\
& = - \mu \int_{\p \OM} \dot{\mathbf{u}} \cdot K \cdot \dot{\mathbf{u}} ds
+ \mu \int_{\p \OM} (\mathbf{u} \cdot \na \mathbf{u}) \cdot K \cdot \dot{\mathbf{u}} ds \\
& = - \mu \int_{\p \OM} \dot{\mathbf{u}} \cdot K \cdot \dot{\mathbf{u}} ds
+ \mu \int_{\p \OM} \mathbf{u}^\bot \times n \cdot \nabla \mathbf{u}^i (K^i \cdot \dot{\mathbf{u}}) ds \\
& = - \mu \int_{\p \OM} \dot{\mathbf{u}} \cdot K \cdot \dot{\mathbf{u}} ds
+ \mu \int_{\p \OM} n \cdot ( \na \mathbf{u}^i \times \mathbf{u}^\bot ) (K^i \cdot \dot{\mathbf{u}}) ds \\
& = -\mu \int_{\p \OM} \dot{\mathbf{u}} \cdot K \cdot \dot{\mathbf{u}} ds
+ \mu \int \div ( ( \na \mathbf{u}^i \times \mathbf{u}^\bot ) (K^i \cdot \dot{\mathbf{u}}) )dx \\
& = - \mu \int_{\p \OM} \dot{\mathbf{u}} \cdot K \cdot \dot{\mathbf{u}} ds
- \mu \int ( \na \mathbf{u}^i \cdot \na \times \mathbf{u}^\bot ) (K^i \cdot \dot{\mathbf{u}}) dx \\
& \quad + \mu \int  \na (K^i \cdot \dot{\mathbf{u}}) \cdot ( \na \mathbf{u}^i \times \mathbf{u}^\bot ) dx \\
& \le - \mu \int_{\p \OM} \dot{\mathbf{u}} \cdot K \cdot \dot{\mathbf{u}} ds
+ C \| \dot{\mathbf{u}} \|_{L^2} \| \na \mathbf{u} \|^2_{L^4}
+ C \| \na \dot{\mathbf{u}} \|_{L^2} \| \na \mathbf{u} \|^2_{L^4} \\
& \le - \mu \int_{\p \OM} \dot{\mathbf{u}} \cdot K \cdot \dot{\mathbf{u}} ds
+ \ep \| \na \dot{\mathbf{u}} \|^2_{L^2} + C(\ep) A^2_2
+ C(\ep) \| \na \mathbf{u} \|^4_{L^4},
\ea\ee
where the symbol $K^i$ denotes the $i$-th row of the matrix $K$.

This combined with (\ref{gb1100}) and (\ref{gb111}) implies
\be\la{gb114}\ba
I_2 & \le - \mu \int_{\p \OM} \dot{\mathbf{u}} \cdot K \cdot \dot{\mathbf{u}} ds
- \mu \| \curl \dot{\mathbf{u}} \|^2_{L^2}
+ 2\ep \Vert \nabla \dot{\mathbf{u}} \Vert_{L^2}^2 \\
& \quad + C(\ep)\left(A^2_1 + A^6_1 \right)+C(\ep) A^2_2 \left(1 + A^4_1 \right).
\ea\ee
We set
\be\la{gb115}\ba
f(t) \triangleq \frac{1}{2} \int \rho|\dot{\mathbf{u}}|^2 dx
+ \int G \p_i \mathbf{u} \cdot \na \mathbf{u}^i dx
+ \int_{\partial \Omega} G (\mathbf{u} \cdot \nabla n \cdot \mathbf{u}) ds.
\ea\ee
By virtue of (\ref{gb11}), (\ref{gb110}), (\ref{gb114}), and (\ref{gb115}), we have
\be\la{gb116}\ba
& \frac{d}{dt} f(t) + \frac{1}{2\nu} \| \dot{G} \|^2_{L^2}
+ \mu \| \curl \dot{\mathbf{u}}\|^2_{L^2}
+ \mu \int_{\p \OM} \dot{\mathbf{u}} \cdot K \cdot \dot{\mathbf{u}} ds \\
& \le 5 \varepsilon \Vert \nabla {\dot{\mathbf{u}}}\Vert _{L^2}^2
+ C(\ep) \| G \|^4_{L^4}
+ C(\ep) \left( A^2_1 + A^6_1 \right) + C(\ep) A^2_2 \left(1 + A^4_1 \right).
\ea\ee
In addition, the boundary condition (\ref{i3}) gives
\be\ba\nonumber
\left( \dot{\mathbf{u}} + (\mathbf{u} \cdot \na n) \times \mathbf{u}^\bot \right) \cdot n = 0 \quad \text{ on } \p \OM.
\ea\ee
Then, we define $\mathbf{v} \triangleq \dot{\mathbf{u}} + (\mathbf{u} \cdot \na n) \times \mathbf{u}^\bot$, which implies that $\mathbf{v} \cdot n = 0$ on $\p \OM$.
By Lemma \ref{dltdc2}, we obtain
\be\ba\la{gb117}
2 \mu \| D(\mathbf{v}) \|^2_{L^2} = 2 \mu \| \div \mathbf{v} \|^2_{L^2}
+ \mu \| \curl \mathbf{v} \|^2_{L^2}
- 2 \mu \int_{\p \OM} \mathbf{v} \cdot D(n) \cdot \mathbf{v} ds.
\ea\ee
Moreover, it follows from Young's inequality that
\be\ba\la{gb118}
2 \mu \| \div \mathbf{v} \|^2_{L^2} + \mu \| \curl \mathbf{v} \|^2_{L^2}
& \le 2 \mu \| \div \dot{\mathbf{u}} \|^2_{L^2} + \mu \| \curl \dot{\mathbf{u}} \|^2_{L^2}
+ C \| \na \dot{\mathbf{u}} \|_{L^2} \| \na \mathbf{u} \|^2_{L^4} \\
& \le \frac{4\mu}{\nu^2} \| \dot{G} \|^2_{L^2} + \mu \| \curl \dot{\mathbf{u}} \|^2_{L^2}
+ \ep \| \na \dot{\mathbf{u}} \|^2_{L^2} \\
& \quad + C \| \na \mathbf{u} \|^2_{L^2} + C(\ep) \| \na \mathbf{u} \|^4_{L^4},
\ea\ee
where in the last inequality we have used the following fact:
\be\ba\nonumber
\| \div \dot{\mathbf{u}} \|^2_{L^2}
\le \frac{2}{\nu^2} \| \dot{G} \|^2_{L^2}
+ C \| \na \mathbf{u} \|^2_{L^2} + C \| \na \mathbf{u} \|^4_{L^4},
\ea\ee
due to (\ref{gb108}) and Young's inequality.

Consequently, when $\nu \ge \nu_1 \ge 8\mu$, we conclude from (\ref{gb117}) and (\ref{gb118}) that
\be\la{gb121}\ba
2 \mu \| D(\mathbf{v}) \|^2_{L^2} + 2 \mu \int_{\p \OM} \mathbf{v} \cdot D(n) \cdot \mathbf{v} ds
& \le \frac{1}{2\nu} \| \dot{G} \|^2_{L^2} + \mu \| \curl \dot{\mathbf{u}}\|^2_{L^2}
+ \ep \| \na \dot{\mathbf{u}} \|^2_{L^2} \\
& \quad + C \| \na \mathbf{u} \|^2_{L^2} + C(\ep) \| \na \mathbf{u} \|^4_{L^4}.
\ea\ee
On the other hand, Young's inequality ensures that
\be\la{gb122}\ba
\mu \int_{\p \OM} \mathbf{v} \cdot K \cdot \mathbf{v} ds
\le \mu \int_{\p \OM} \dot{\mathbf{u}} \cdot K \cdot \dot{\mathbf{u}} ds
+\ep \| \na \dot{\mathbf{u}} \|^2_{L^2}
+ C(\ep)\left( \| \sqrt{\n} \dot{\mathbf{u}}\|^2_{L^2} + \| \na \mathbf{u} \|^4_{L^4} \right).
\ea\ee
Combining (\ref{gb1100}), (\ref{gb116}), (\ref{gb121}), and (\ref{gb122}) yields
\be\la{gb123}\ba
& \frac{d}{dt} f(t) + 2 \mu \| D(\mathbf{v}) \|^2_{L^2}
+ \mu \int_{\p \OM} \mathbf{v} \cdot ( K + 2D(n) ) \cdot \mathbf{v} ds \\
& \le 7 \varepsilon \Vert \nabla {\dot{\mathbf{u}}}\Vert _{L^2}^2
+ C(\ep) \| G \|^4_{L^4}
+ C(\ep) \left( A^2_1 + A^6_1 \right) + C(\ep) A^2_2 \left(1 + A^4_1 \right).
\ea\ee
Furthermore, from the definition of $\mathbf{v}$ and Lemma \ref{dltdc1}, we derive
\be\la{gb124}\ba
\| \na \dot{\mathbf{u}} \|^2_{L^2}
& \le C \| \na \mathbf{v} \|^2_{L^2} + C \| \na \mathbf{u} \|^4_{L^4} \\
& \le C \left( 2 \| D(\mathbf{v}) \|^2_{L^2}
+ \int_{\p \OM} \mathbf{v} \cdot ( K+2D(n) ) \cdot \mathbf{v} ds \right)
+ C \| \na \mathbf{u} \|^4_{L^4},
\ea\ee
which together with (\ref{gb1100}) and (\ref{gb123}) shows that
\be\la{gb125}\ba
& \frac{d}{dt} f(t) + 2 \mu \| D(\mathbf{v}) \|^2_{L^2}
+ \mu \int_{\p \OM} \mathbf{v} \cdot ( K + 2D(n) ) \cdot \mathbf{v} ds \\
& \le C \ep \left( 2 \mu \| D(\mathbf{v}) \|^2_{L^2}
+ \mu \int_{\p \OM} \mathbf{v} \cdot ( K+2D(n) ) \cdot \mathbf{v} ds \right)
+ C(\ep) \| G \|^4_{L^4} \\
& \quad + C(\ep) \left( A^2_1 + A^6_1 \right) + C(\ep) A^2_2 \left(1 + A^4_1 \right).
\ea\ee
Therefore, choosing $\ep$ suitably small and multiplying $\si$, we have
\be\la{gb126}\ba
& \frac{d}{dt} ( \si f(t) ) + \mu \si \| D(\mathbf{v}) \|^2_{L^2}
+ \frac{\mu}{2} \si \int_{\p \OM} \mathbf{v} \cdot ( K + 2D(n) ) \cdot \mathbf{v} ds \\
& \le \si' f(t) + C \si \| G \|^4_{L^4}
+ C \left( A^2_1 + A^6_1 \right) + C A^2_2 \left(1 + A^4_1 \right).
\ea\ee
Moreover, in view of (\ref{zdc03}), (\ref{zdc318}), (\ref{zdc04}), and Young's inequality, it holds that
\be\la{gb127}\ba
& \left| \int G \p_i \mathbf{u} \cdot \na \mathbf{u}^i dx
+ \int_{\partial \Omega} G (\mathbf{u} \cdot \nabla n \cdot \mathbf{u}) ds \right| \\
& \le C \| G \|_{L^4} \| \na \mathbf{u} \|^2_{L^{\frac{8}{3}}}
+ C \| G\|_{H^1} \| \na \mathbf{u} \|^2_{L^2} \\
& \le C \| G \|_{H^1} \left( A_1^{\frac{3}{2}} A_2^{\frac{1}{2}}
+ 1 + A^2_1 \right) \\
& \le C \left( A_1 + A_2 \right)
\left( A_1^{\frac{3}{2}} A_2^{\frac{1}{2}} + 1 + A^2_1 \right) \\
& \le \frac{1}{4} A^2_2 + C(1 + A^6_1),
\ea\ee
which together with (\ref{gb115}) implies
\be\la{gb128}\ba
\frac{1}{4} A^2_2 \le f(t) + C(1 + A^6_1).
\ea\ee
Integrating (\ref{gb126}) over $(0,T)$ and applying
(\ref{zdc313}), (\ref{zdc318}), (\ref{zdc04}), and (\ref{gb128}) lead to
\be\la{gb130}\ba
& \sup_{0 \le t \le T} \si \int \n| \dot{\mathbf{u}} |^2 dx
+ \int_0^T \si \left( \| D(\mathbf{v}) \|^2_{L^2} + \int_{\p \OM} \mathbf{v} \cdot ( K + 2D(n) ) \cdot \mathbf{v} ds \right) dt \\
& \le C \nu^C + C \int_0^T \| G \|^2_{L^2} \| G \|^2_{H^1} dt \\
& \le C \nu^C.
\ea\ee
Combining this with (\ref{gb124}), (\ref{gb1100}), (\ref{zdc1}), (\ref{zdc02}), and (\ref{zdc04}), we arrive at
\be\ba\nonumber
\int_0^T \si \| \na \dot{\mathbf{u}} \|^2_{L^2}
\le C \int_0^T \si \left( \| D(\mathbf{v}) \|^2_{L^2} + \| \na \mathbf{u} \|^4_{L^4}
+ \int_{\p \OM} \mathbf{v} \cdot ( K + 2D(n) ) \cdot \mathbf{v} ds \right) dt
\le C \nu^C,
\ea\ee
which together with (\ref{gb130}) yields (\ref{gb01}).

Finally, we prove (\ref{gb002}).
By virtue of (\ref{zdc1}), (\ref{zdc02}), (\ref{zdc04}) and $\div \mathbf{u}_0=0$, it holds that
\be\la{gb132}\ba
\sup_{0 \le t \le T} A^2_1 + \int_0^T ( A^2_1 + A^2_2 ) dt \le C.
\ea\ee
Then, from (\ref{zdc318}), (\ref{gb126}), (\ref{gb127}), (\ref{gb128}), and (\ref{gb132}), we deduce that
\be\la{gb133}\ba
& \frac{d}{dt} ( \si f(t) ) + \mu \si \| D(\mathbf{v}) \|^2_{L^2}
+ \frac{\mu}{2} \si \int_{\p \OM} \mathbf{v} \cdot ( K + 2D(n) ) \cdot \mathbf{v} ds \\
& \le \si' f(t) + C \si \left( \| \sqrt{\n} \dot{\mathbf{u}} \|_{L^2} + \| \na \mathbf{u} \|_{L^2} \right)^4
+ C \left( A^2_1 + A^6_1 \right) + C A^2_2 \left(1 + A^4_1 \right) \\
& \le C \si' + C \si A^4_2 + C A^2_2 + C A^2_1 \\
& \le C \si f(t) A^2_2 + C \si' + C A^2_2 + C A^2_1.
\ea\ee
Hence, applying Gr\"onwall's inequality to (\ref{gb133}), we obtain after using (\ref{gb128}) and (\ref{gb132}) that
\be\la{gb134}\ba
& \sup_{0 \le t \le T} \si \int \rho|\dot{\mathbf{u}}|^2dx 
+ \int_0^T \si \left( \| D(\mathbf{v}) \|^2_{L^2} + \int_{\p \OM} \mathbf{v} \cdot ( K + 2D(n) ) \cdot \mathbf{v} ds \right) dt
\le C.
\ea\ee
In addition, it follows from (\ref{gb124}), (\ref{gb134}), (\ref{gb132}), and (\ref{gb1100}) that
\be\ba\nonumber
\int_0^T \si \| \na \dot{\mathbf{u}} \|^2_{L^2}
\le C \int_0^T \si \left( \| D(\mathbf{v}) \|^2_{L^2} + \| \na \mathbf{u} \|^4_{L^4}
+ \int_{\p \OM} \mathbf{v} \cdot ( K + 2D(n) ) \cdot \mathbf{v} ds \right) dt
\le C,
\ea\ee
which together with (\ref{gb134}) gives (\ref{gb002}) and completes the proof of Lemma \ref{zdcgl1}.
\end{proof}

Based on the uniform estimates (\ref{zdc07}) and (\ref{gb01}),
we obtain the following exponential decay result.
The proof follows arguments similar to those in \cite[Lemma 4.2]{LX4}.
\begin{lemma}\la{zdce}
For any $s\in [1,\infty)$, there exist positive constants $C$, $D_0$, and $\nu_2$ depending only on
$s$, $\ga$, $\mu$, $\|\n_0\|_{L^1 \cap L^\infty}$, $\| \mathbf{u}_0 \|_{H^1}$, and $K$ such that if $\nu \ge \nu_2$, then for any $t \ge 1$,
\be\la{zdce01}\ba
\| \n-\ol{\n_0}\|_{L^s} + \| \na \mathbf{u} \|_{L^s}
+ \| \sqrt{\n} \dot{\mathbf{u}} \|_{L^2} \le C e^{-\alpha_0 t},
\ea\ee
where $\alpha_0=\frac{D_0}{\nu}$.
\end{lemma}

\begin{lemma}\la{gz2}
For $3<q<\infty$, there exists a positive constant $C$ depending only on 
$T$, $q$, $\ga$, $\mu$, $\nu$, $\| \rho_0 \|_{W^{1,q}}$,
$\| \mathbf{u}_0 \|_{H^1}$, and $K$ such that
\be\la{gz201}\ba
&\sup_{0\le t\le T}\left(\norm[W^{1,q}]{ \rho}+\| \mathbf{u} \|_{H^1}
+ t\| \sqrt{\n} \mathbf{u}_t \|^2_{L^2} + t\| \mathbf{u} \|^2_{H^2}
+ \| \n_t \|_{L^2} \right) \\
& + \int_0^T \left( \|\nabla^2 \mathbf{u}\|^{2}_{L^2}
+ \|\nabla^2 \mathbf{u}\|^{(q+1)/q}_{L^q}
+ t \|\nabla^2 \mathbf{u}\|_{L^q}^2
+ \| \sqrt{\n} \mathbf{u}_t \|^2_{L^2} + t\| \mathbf{u}_t\|_{H^1}^2 \right) dt
\le C.
\ea\ee
\end{lemma}
\begin{proof}
First, differentiating $(\ref{ns})_1$ with respect to $x$ and multiplying the
resulting equation by $q |\nabla\n|^{q-2} \na \n $, we derive
\be\la{s432}\ba
& (|\nabla\n|^q)_t + \text{div}(|\nabla\n|^q \mathbf{u})+ (q-1)|\nabla\n|^q\text{div} \mathbf{u} \\
& + q |\nabla\n|^{q-2} \p_i\n \p_i \mathbf{u}^j \p_j\n +
q\n|\nabla\n|^{q-2}\p_i\n  \p_i \text{div}\mathbf{u} = 0.
\ea\ee
Integrating (\ref{s432}) over $\OM$ and applying the boundary condition (\ref{i3}) yield
\be\la{s433}\ba
\frac{d}{dt} \| \na \n \|_{L^q}
&\le C \| \na \mathbf{u} \|_{L^\infty} \| \na \n \|_{L^q}
+ C \| \na^2 \mathbf{u} \|_{L^q} \\
& \le C \left( 1 + \| \na \mathbf{u} \|_{L^\infty} + \| \n \dot{\mathbf{u}} \|_{L^q} \right)
\| \na \n \|_{L^q} + C \| \n \dot{\mathbf{u}} \|_{L^q},
\ea\ee 
where we have used the following estimate:
\be\la{s434}\ba
\|\na^2 \mathbf{u}\|_{L^q}
& \le C \left( \| \div \mathbf{u} \|_{W^{1,q}}
+ \| \curl \mathbf{u} \|_{W^{1,q}} + \| \mathbf{u} \|_{L^q} \right) \\
& \le C + C \left( \| \na \mathbf{u} \|_{L^q}
+ \| \na \div \mathbf{u} \|_{L^q}
+ \| \na \curl \mathbf{u} \|_{L^q} \right) \\
& \le C + C \left( \| \na G \|_{L^q}
+ \| \na \n \|_{L^q}
+ \| \n \dot{\mathbf{u}} \|_{L^q} \right) \\
& \le C \left( 1+\| \n \dot{\mathbf{u}} \|_{L^q} + \| \na \n \|_{L^q} \right),
\ea\ee
due to (\ref{dc1}), (\ref{zdc316}), (\ref{zdc317}), and (\ref{zdc04}).

From (\ref{zdc316}), (\ref{zdc317}), (\ref{zdc04}), and Sobolev embedding, we deduce that
\be\ba\nonumber
& \| \div \mathbf{u} \|_{L^\infty}+\| \curl \mathbf{u} \|_{L^\infty} \\
& \le C \left( \| G \|_{L^\infty} + \| P-\ol{P} \|_{L^\infty} \right)
+ \| \curl \mathbf{u} \|_{L^\infty} \\
& \le C + C \left( \| \na G \|_{L^q}
+ \| \curl \mathbf{u} \|_{L^2} + \| \na \curl \mathbf{u} \|_{L^q} \right) \\
& \le C \left( 1 + \| \n \dot{\mathbf{u}} \|_{L^q} \right),
\ea\ee
which together with (\ref{s434}), (\ref{zdc04}) and Lemma \ref{bkm1} gives
\be\ba\nonumber
\|\na \mathbf{u}\|_{L^\infty} 
& \le C \left( \|\div \mathbf{u} \|_{L^\infty}
+ \|\curl \mathbf{u}\|_{L^\infty} \right) \log \left(e+ \|\na^2 \mathbf{u}\|_{L^q} \right)+ C\|\na \mathbf{u}\|_{L^2}+C \\
& \le C \left( 1 + \| \n \dot{\mathbf{u}} \|_{L^q} \right)
\log \left(e + \| \na \n \|_{L^q} + \| \n \dot{\mathbf{u}} \|_{L^q} \right) \\
& \le C \left( 1 + \| \n \dot{\mathbf{u}} \|_{L^q} \right)
\log \left(e + \| \na \n \|_{L^q} \right)
+ C \| \n \dot{\mathbf{u}} \|^{1+1/q}_{L^q}.
\ea\ee
Combining this with (\ref{s433}) shows that
\be\la{s437}\ba
\frac{d}{dt} \log( e + \| \na \n \|_{L^q} )
& \le C \left( 1 + \| \n \dot{\mathbf{u}} \|_{L^q} \right)
\log \left(e + \| \na \n \|_{L^q} \right)
+ C \| \n \dot{\mathbf{u}} \|^{1+1/q}_{L^q}.
\ea\ee
In view of (\ref{ewgn01}), (\ref{pt1}), and H\"{o}lder's inequality, it holds that
\be\ba\nonumber
\| \rho \dot{\mathbf{u}} \|_{L^q} 
& \le C\| \rho \dot{\mathbf{u}} \|_{L^2}^{2(q-1)/(q^2-2)}
\| \dot{\mathbf{u}} \|_{L^{q^2}}^{q(q-2)/(q^2-2)} \\
& \le C\| \rho \dot{\mathbf{u}} \|_{L^2}^{2(q-1)/(q^2-2)}
\| \dot{\mathbf{u}} \|_{H^1}^{q(q-2)/(q^2-2)} \\
& \le C\| \rho^{1/2} \dot{\mathbf{u}} \|_{L^2}
+ C\| \rho \dot{\mathbf{u}} \|_{L^2}^{2(q-1)/(q^2-2)}
\| \na \dot{\mathbf{u}} \|_{L^2}^{q(q-2)/(q^2-2)},
\ea\ee
which along with (\ref{gb01}) and (\ref{pt1}) yields
\be\la{s439}\ba
&\int_0^T \left(\| \rho \dot{\mathbf{u}} \|^{1+1 /q}_{L^q}
+ t \| \dot{\mathbf{u}} \|^2_{H^1} \right) dt \\
&\le C+C \int_0^T\left( \| \rho^{1/2} \dot{\mathbf{u}} \|_{L^2}^2
+ t\| \na \dot{\mathbf{u}} \|_{L^2}^2
+ t^{-(q^3-q^2-2p)/(q^3-q^2-2p+2)} \right) dt \\ 
&\le C.
\ea\ee
Thus, applying Gr\"onwall's inequality to (\ref{s437}) and using (\ref{s439}), we derive
\be\la{s4310}\ba
\sup_{0 \le t \le T} \| \n \|_{W^{1,q}} \le C,
\ea\ee
which together with (\ref{gb01}), (\ref{s434}) and (\ref{s439}) implies
\be\la{s4311}\ba
\sup_{0\le t\le T} t \| \na^2 \mathbf{u} \|^2_{L^2}
+ \int_0^T \left( \|\nabla^2 \mathbf{u} \|_{L^2}^2
+ \| \nabla^2 \mathbf{u} \|^{(q+1)/q}_{L^q}
+ t \|\nabla^2 \mathbf{u} \|_{L^q}^2 \right) dt
\le C.
\ea\ee
On the other hand, it follows from $(\ref{ns})_1$, (\ref{zdc38}), (\ref{zdc04}), and (\ref{s4310}) that
\bnn  
\| \n_t\|_{L^2}\le
C\| \mathbf{u} \|_{L^{2q/(q-2)}}\|\nabla \n \|_{L^q}
+ C\|\n\|_{L^\infty} \| \nabla \mathbf{u} \|_{L^2} \le C,
\enn
which gives
\be\la{s43110}\ba
\sup_{0\le t\le T} \| \n_t\|_{L^2} \le C.
\ea\ee

Finally, by virtue of (\ref{zdc38}), (\ref{zdc04}) and H\"older's inequality, we have
\be \la{s4312}\ba 
\int\rho|\mathbf{u}_t|^2dx 
&\le \int\rho| \dot{\mathbf{u}} |^2dx+\int \n |\mathbf{u} \cdot \na \mathbf{u}|^2dx \\
&\le \int\rho| \dot{\mathbf{u}} |^2dx + C \| \mathbf{u} \|_{L^4}^2 \| \na \mathbf{u} \|_{L^4}^2 \\ 
&\le \int\rho| \dot{\mathbf{u}} |^2dx + C\| \na^2 \mathbf{u} \|_{L^2}^2 + C,
\ea\ee
and
\be\la{s4313}\ba 
\|\nabla \mathbf{u}_t\|_{L^2}^2 
&\le \| \nabla \dot{\mathbf{u}} \|_{L^2}^2+ \| \nabla(\mathbf{u}\cdot\nabla \mathbf{u})\|_{L^2}^2 \\
&\le \|\nabla \dot{\mathbf{u}} \|_{L^2}^2+ \| \mathbf{u} \|_{L^{2q/(q-2)}}^2\|\nabla^2\mathbf{u} \|_{L^q}^2
+ \| \nabla \mathbf{u} \|_{L^4}^4 \\ 
&\le \|\nabla \dot{\mathbf{u}} \|_{L^2}^2 + C\|\nabla^2 \mathbf{u} \|_{L^q}^2+ \| \nabla \mathbf{u} \|_{L^4}^4.
\ea\ee 
Therefore, we conclude from (\ref{pt1}), (\ref{gb01}), (\ref{s4311}), (\ref{s4312}) and (\ref{s4313}) that
\be\la{s4314}\ba
\sup_{0\le t\le T} t \| \sqrt{\n} \mathbf{u}_t \|^2_{L^2}
+ \int_0^T \| \sqrt{\n} \mathbf{u}_t \|^2_{L^2}
+ t\|\mathbf{u}_t\|_{H^1}^2 dt \le C.
\ea\ee
This combined with (\ref{s4310}), (\ref{s4311}), and (\ref{s43110}) yields (\ref{gz201}) and finishes the proof of Lemma \ref{gz2}.
\end{proof}

From now on, we assume that the initial data $(\n_0,\mathbf{u}_0)$ satisfy (\ref{csol1}) and the compatibility condition (\ref{csol2}).
\begin{lemma}\la{c21}
There exists a positive constant $C$ depending only on  
$T$, $\ga$, $\mu$, $\nu$, $\|\n_0\|_{L^\infty}$, $\| \mathbf{u}_0 \|_{H^1}$, $\| g \|_{L^2}$, and $K$ such that
\be\la{c411}\ba
\sup_{0\le t\le T}
\int \n |\dot{\mathbf{u}}|^2dx
+\int_0^{T} \int |\na \dot{\mathbf{u}}|^2 dx dt \le C.
\ea\ee
\end{lemma}
\begin{proof}
Taking into account the compatibility condition (\ref{csol2}), we define
\be\ba\nonumber
\sqrt{\n} \dot{\mathbf{u}}(x,t=0)=g(x).
\ea \ee
Integrating (\ref{gb125}) over $(0,T)$, choosing $\ep$ sufficiently small,
and applying (\ref{gb128}), (\ref{gb124}), and (\ref{zdc04}), we obtain (\ref{c411}).
\end{proof}

To extend the local classical solution globally, we require the following higher-order estimates.
Since the proofs of these estimates are analogous to those in \cite{CL,LX3}, we omit them.
\begin{lemma}\la{c26}
There exists a positive constant C depending only on
$T$,\ $\mu$,\ $\nu$,\ $\ga$, $K$, $\| \mathbf{u}_0 \|_{H^2}$, $\| \rho_0 \|_{W^{2,q}}$, $\| P(\rho_0) \|_{W^{2,q}}$, and $\|g\|_{L^2}$ such that
\be\la{c421}\ba
& \sup_{0\le t\le T}\left(\norm[W^{1,q}]{ \rho} + \| \mathbf{u}\|_{H^2} +\| \sqrt{\n} \mathbf{u}_t\|_{L^2} + \| \n_t\|_{L^2} \right) \\
& + \int_0^T \left( \|\nabla^2 \mathbf{u}\|_{L^q}^2 + \| \na \mathbf{u}_t\|^2_{L^2} \right) dt \le C,
\ea\ee
\be\la{c431}\ba 
& \sup_{t\in[0,T]} \left(\norm[H^2]{\rho} + \norm[H^2]{P(\rho)}+ \|\n_t\|_{H^1}+\|P_t\|_{H^1} \right) \\
& + \int_0^T\left(\| \na^3 \mathbf{u}\|^2_{L^2}+\|\n_{tt}\|_{L^2}^2 +\|P_{tt}\|_{L^2}^2\right)dt \le C,
\ea\ee
\be\la{c442}\ba
\sup\limits_{0\le t\le T} t^{1/2} \left(\| \na \mathbf{u}_t\|_{L^2}
+ \| \na^3 \mathbf{u}\|_{L^2} \right) 
+ \int_0^T t\left(\| \sqrt{\n} \mathbf{u}_{tt}\|^2_{L^2}
+ \| \na^2 \mathbf{u}_t\|^2_{L^2}\right)dt
\le C,
\ea\ee
\be\la{c451}\ba
\sup_{0\leq t\leq T}\left( \|\nabla^2 \n\|_{L^q}
+ \|\nabla^2 P  \|_{L^q}\right ) \leq C,
\ea\ee
\be\la{c461}\ba
\sup_{0\leq t\leq T} t\left(\| \sqrt{\n} \mathbf{u}_{tt}\|_{L^2} +  \|\na^3 \mathbf{u} \|_{L^q} + \|\na^2
\mathbf{u}_t \|_{L^2} \right) + \int_{0}^T  t^2 \|\nabla \mathbf{u}_{tt}\|_{L^2}^2 dt\leq C.
\ea\ee
\end{lemma}

\section{Proofs of Theorems \ref{th0}--\ref{th3}}
In this section, we focus on proving the main results.
We first establish the global existence of classical solutions to problem (\ref{ns})--(\ref{i3}) for $\nu \ge \nu_1$, with $\nu_1$ given by (\ref{zdc78}).
Standard compactness arguments then yield the global existence of strong and weak solutions.

Proof of Theorem \ref{th2}.
Let $(\n_0,\mathbf{u}_0)$ be the initial data in Theorem \ref{th2}, satisfying (\ref{csol1}) and (\ref{csol2}).
For any $\de \in (0,1)$, define
\be\la{czbj1}\ba
\n_0^{\de} \triangleq \n_0 + \de,
\ea\ee
which together with (\ref{csol1}) gives
\be\la{czbj2}\ba
0<\de \le \n_0^\de \le \| \n_0 \|_{L^\infty}+1,
\ea\ee
and
\be\ba\nonumber
\lim\limits_{\de\rightarrow 0} \|\n_0^\de-\n_0\|_{ W^{2,q}} = 0.
\ea\ee
In addition, we set
\be\la{czbj3}\ba
g^\de \triangleq ( \n^\de_0 )^{-1/2} \left( - \mu \Delta \mathbf{u}_0 - (\mu + \lm )\nabla \div \mathbf{u}_0
+ \nabla P(\n^\de_0) \right),
\ea\ee
The compatibility condition (\ref{csol2}) immediately implies
\be\ba\nonumber
g^\de = ( \n^\de_0 )^{-1/2} ( \n_0 )^{1/2} g,
\ea\ee
which along with (\ref{czbj1}) yields
\be\la{czbj4}\ba
\| g^\de \|_{L^2} \le \| g \|_{L^2}.
\ea\ee

Therefore, by the local existence Lemma \ref{lct}, for problem (\ref{ns})--(\ref{i3}) with initial data $(\rho^\de_0,\mathbf{u}^\de_0)$, there exists $T_\de>0$ such that the problem has a unique classical solution $(\n^\de,\mathbf{u}^\de)$ on $\OM \times (0,T_\de]$.
From Lemmas \ref{l7}, \ref{c21} and \ref{c26}, we can extend 
the local solution $(\n^\de,\mathbf{u}^\de)$ to $\OM \times (0,T]$ for any $T>0$, provided $\nu \ge \nu_1$.
Furthermore, (\ref{czbj2}) and (\ref{czbj4}) guarantee that all constants $C$ in Lemmas \ref{l7}, \ref{c21} and \ref{c26} are independent of $\de$.
Letting $\de \to 0$ and applying standard arguments (see \cite{LX3,LZZ}), we obtain that the problem  (\ref{ns})--(\ref{i3}) 
has a global classical solution $(\n,\mathbf{u})$ satisfying (\ref{csol4}) under the condition $\nu \ge \nu_1$.
The proof of uniqueness of $(\n,\mathbf{u})$ satisfying (\ref{csol4}) is similar to \cite{Ge},
which completes the proof of Theorem~\ref{th2}.

The proofs of Theorems \ref{th0} and \ref{th1} follow by standard compactness arguments as in \cite{F,L2}, and are omitted here.

Proof of Theorem \ref{th01}.
For any $\nu \ge \nu_1$, we deduce from (\ref{wsol2}), (\ref{zdc1}), (\ref{zdc02}), (\ref{zdc04}), (\ref{gb002}), and Poincar\'e's inequality
that $\{\n^{\nu}\}_\nu$ is bounded in $L^\infty(\OM \times (0,\infty))$ and
$\{\mathbf{u}^{\nu}\}_\nu$ is bounded in $L^\infty(0,\infty;H^1) \cap L^2(0,\infty;H^1) \cap H^1(\tau,\infty;L^2)$ for any $\tau>0$.
Consequently, without loss of generality, we can assume that
there exists a subsequence $(\n^n,\mathbf{u}^n)$ of $(\n^{\nu},\mathbf{u}^{\nu})$ and
$\n \in L^\infty(\OM \times (0,\infty)),
\mathbf{u} \in L^\infty(0,\infty;H^1) \cap L^2(0,\infty;H^1)$ such that
\be\ba\nonumber
\begin{cases}
\n^n \rightharpoonup \n  \mbox{ weakly * in } L^\infty(\OM \times (0,\infty)),\\
\mathbf{u}^n \rightharpoonup \mathbf{u}  \mbox{ weakly * in } \ L^\infty(0,\infty;H^1) \cap L^2(0,\infty;H^1), \\
\mathbf{u}^n \to \mathbf{u}  \mbox{ strongly  in } \ L^\infty(\tau,T;L^p),
\end{cases}
\ea\ee 
for any $1 \le p < \infty$ and $\tau>0$.

Letting $n \to \infty$ and arguing as in \cite{LX3}, we conclude that $(\rho,\mathbf{u})$ satisfies (\ref{isol1}), (\ref{isol2}) and (\ref{isol3}).

Next, we prove that if the initial data $(\rho_0,\mathbf{u}_0)$ further satisfy (\ref{insc1}),
then the entire sequence $(\n^\nu,\mathbf{u}^\nu)$ converges to the unique global strong solution of (\ref{isol1}) and satisfies (\ref{insc3}).

First, under the hypotheses of Theorem \ref{th0},
in view of (\ref{dltdc01}), (\ref{dltdc02}), and the logarithmic interpolation inequality (\ref{logbd1}),
and following arguments analogous to those in \cite{GWX,WG},
we can derive that the system (\ref{isol1}) admits a unique global axisymmetric strong solution $(\hat{\n},\hat{\mathbf{u}},\hat{\pi})$ satisfying (\ref{insc3}) for any $0<T<\infty$.

We now prove that for any $0<T<\infty$, $\mathbf{u}=\hat{\mathbf{u}}$ and $\n=\hat{\n}$ a.e. in $\OM \times (0,T)$.
By virtue of (\ref{isol1}) and (\ref{isol2}), it holds that
\be\ba\la{uins3}
\frac{1}{2}\int \n|\mathbf{u}|^2 dx + \mu \int_0^t\int|\curl \mathbf{u}|^2 dx d\tau + \mu \int_0^t \int_{\p \OM} \mathbf{u} \cdot K \cdot \mathbf{u} ds d\tau=
\frac{1}{2}\int \n_0 |\mathbf{u}_0|^2 dx.
\ea\ee
In addition, we multiply $(\ref{isol1})_2$ by $\hat{\mathbf{u}}$ and integrate over $\OM \times (0,t)$ to derive
\be\ba\la{uins4}
& \int \n \mathbf{u} \cdot \hat{\mathbf{u}} dx
+ \mu \int_0^t\int  \curl \mathbf{u} \cdot \curl \mathbf{\hat{u}} dx d\tau
+\mu \int_0^t \int_{\p \OM} \hat{\mathbf{u}} \cdot K \cdot \mathbf{u} ds d\tau \\
& =\int \n_0 |\mathbf{u}_0|^2 dx + \int_0^t \int \n \mathbf{u} \left( \hat{\mathbf{u}}_t +\mathbf{u} \cdot \na \hat{\mathbf{u}} \right) dx d\tau.
\ea\ee
Then we rewrite $(\ref{isol1})_2$ as
\be\la{uins5}\ba
\n \hat{ \mathbf{u}}_t + \n  \mathbf{u} \cdot \na \hat{ \mathbf{u}} -\mu \Delta \hat{ \mathbf{u}} +\na \hat{\pi}
=(\n-\hat{\n}) \left( \hat{ \mathbf{u}}_t + \hat{ \mathbf{u}} \cdot \na \hat{ \mathbf{u}} \right) +\n ( \mathbf{u}-\hat{ \mathbf{u}}) \cdot \na \hat{ \mathbf{u}}.
\ea\ee
Multiplying (\ref{uins5}) by $ \mathbf{u}$ and integrating by parts lead to
\be\la{uins6}\ba
&\int_0^t \int \left( \n \hat{ \mathbf{u}}_t + \n  \mathbf{u} \cdot \na \hat{ \mathbf{u}} \right) \cdot  \mathbf{u} dx d\tau
+ \mu \int_0^t\int   \curl \mathbf{u} \cdot \curl \hat{\mathbf{u}} dx d\tau
+\mu\int_0^t \int_{\p \OM} \mathbf{u} \cdot K \cdot \hat{\mathbf{u}} ds d\tau \\
&=\int_0^t \int \left( (\n-\hat{\n}) \left( \hat{ \mathbf{u}}_t + \hat{ \mathbf{u}} \cdot \na \hat{ \mathbf{u}} \right) \cdot  \mathbf{u} + \n ( \mathbf{u}-\hat{ \mathbf{u}}) \cdot \na \hat{ \mathbf{u}} \cdot  \mathbf{u} \right) dx d\tau.
\ea\ee
By adding (\ref{uins4}) and (\ref{uins6}), we have
\be\ba\la{uins7}
&\int \n  \mathbf{u} \cdot \hat{ \mathbf{u}} dx 
+ 2\mu \int_0^t\int   \curl \mathbf{u} \cdot \curl \hat{\mathbf{u}} dx d\tau
+ \mu\int_0^t \int_{\p \OM} \left( \hat{\mathbf{u}} \cdot K \cdot \mathbf{u} + \mathbf{u} \cdot K \cdot \hat{\mathbf{u}} \right) ds d\tau \\
& = \int \n_0 |\mathbf{u}_0|^2 dx
+ \int_0^t \int \left( (\n-\hat{\n}) \left( \hat{ \mathbf{u}}_t + \hat{ \mathbf{u}} \cdot \na \hat{ \mathbf{u}} \right) \cdot  \mathbf{u} 
+\n ( \mathbf{u}-\hat{ \mathbf{u}}) \cdot \na \hat{ \mathbf{u}} \cdot  \mathbf{u} \right) dx d\tau.
\ea\ee
Moreover, multiplying (\ref{uins5}) by $\hat{\mathbf{u}}$ yields
\be\ba\la{uins8}
&\frac{1}{2}\int \n|\hat{\mathbf{u}}|^2 dx + \mu \int_0^t\int |\curl \hat{\mathbf{u}}|^2 dx d\tau
+ \mu \int_0^t \int_{\p \OM} \hat{\mathbf{u}} \cdot K \cdot \hat{\mathbf{u}} ds d\tau \\
&=\frac{1}{2}\int \n_0 |\mathbf{u}_0|^2 dx
+ \int_0^t \int \left( (\n-\hat{\n}) \left( \hat{\mathbf{u}}_t
+ \hat{\mathbf{u}} \cdot \na \hat{\mathbf{u}} \right) \cdot \hat{\mathbf{u}}
+ \n (\mathbf{u}-\hat{\mathbf{u}}) \cdot \na \hat{\mathbf{u}} \cdot \hat{\mathbf{u}} \right) dx d\tau.
\ea\ee
We add (\ref{uins3}) to (\ref{uins8}), subtract (\ref{uins7}), and then use (\ref{insc3}) to obtain
\be\ba\nonumber
&\frac{1}{2}\int \n |\mathbf{u}-\hat{\mathbf{u}}|^2 dx
+ \mu \int_0^t \int  |\curl (\mathbf{u} - \hat{\mathbf{u}})|^2 dx d\tau
+ \mu \int_0^t \int_{\p \OM} (\mathbf{u} - \hat{\mathbf{u}}) \cdot K \cdot (\mathbf{u} - \hat{\mathbf{u}}) ds d\tau \\
&\le \int_0^t \int \left( (\n-\hat{\n}) \left( \hat{\mathbf{u}}_t + \hat{\mathbf{u}} \cdot \na \hat{\mathbf{u}} \right)
\cdot ( \hat{\mathbf{u}} - \mathbf{u} ) +\n (\mathbf{u}-\hat{\mathbf{u}}) \cdot \na \hat{\mathbf{u}} \cdot ( \hat{\mathbf{u}} - \mathbf{u} ) \right) dx d\tau \\
& \le C \int_0^t \left( \| \n-\hat{\n} \|_{L^2} \| \hat{\mathbf{u}}_t \|_{L^4} \| \hat{\mathbf{u}} - \mathbf{u} \|_{L^4}
+ \| \n-\hat{\n} \|_{L^2} \| \hat{\mathbf{u}} \|_{L^8} \| \na \hat{\mathbf{u}} \|_{L^4} \| \hat{\mathbf{u}} - \mathbf{u} \|_{L^8} \right) d\tau \\
& \quad + C \int_0^t  \| \na \hat{\mathbf{u}} \|_{L^\infty} \| \sqrt{\n} (\mathbf{u}-\hat{\mathbf{u}}) \|^2_{L^2} d\tau \\
& \le C \int_0^t \| \n-\hat{\n} \|_{L^2} \left( \| \hat{\mathbf{u}}_t \|_{L^4} + \| \na \hat{\mathbf{u}} \|_{L^4} \right)
\| \na ( \mathbf{u} - \hat{\mathbf{u}} ) \|_{L^2} d\tau
+ C \int_0^t  \| \na \hat{\mathbf{u}} \|_{L^\infty} \| \sqrt{\n} (\mathbf{u}-\hat{\mathbf{u}}) \|^2_{L^2} d\tau \\
& \le \frac{\mu}{2} \int_0^t \int  |\curl (\mathbf{u} - \hat{\mathbf{u}})|^2 dx d\tau
+ \frac{\mu}{2} \int_0^t \int_{\p \OM} (\mathbf{u} - \hat{\mathbf{u}}) \cdot K \cdot (\mathbf{u} - \hat{\mathbf{u}}) ds d\tau \\
& \quad + C \int_0^t \| \n-\hat{\n} \|^2_{L^2} \left( \| \hat{\mathbf{u}}_t \|^2_{L^4} + \| \na \hat{\mathbf{u}} \|^2_{L^4} \right) d\tau
+C \int_0^t \| \na \hat{\mathbf{u}} \|_{L^\infty} \| \sqrt{\n} (\mathbf{u}-\hat{\mathbf{u}}) \|^2_{L^2} d\tau,
\ea\ee
where in the last inequality we have used Young's inequality and the following estimate:
\be\ba\la{uins9}
\| \na ( \mathbf{u} - \hat{\mathbf{u}} ) \|^2_{L^2}
& \le \Lambda \left( 2\| D(\mathbf{u} - \hat{\mathbf{u}}) \|^2_{L^2}
+ \int_{\p \OM} (\mathbf{u} - \hat{\mathbf{u}}) \cdot (K+2D(n)) \cdot (\mathbf{u} - \hat{\mathbf{u}}) ds \right) \\
& = \Lambda \left( \| \curl(\mathbf{u} - \hat{\mathbf{u}}) \|^2_{L^2}
+ \int_{\p \OM} (\mathbf{u} - \hat{\mathbf{u}}) \cdot K \cdot (\mathbf{u} - \hat{\mathbf{u}}) ds \right),
\ea\ee
due to (\ref{dltdc01}) and (\ref{dltdc02}).

Therefore, we have
\be\ba\la{uins10}
& \| \sqrt{\n} (\mathbf{u}-\hat{\mathbf{u}}) \|^2_{L^2}
+ \mu \int_0^t \| \curl(\mathbf{u} - \hat{\mathbf{u}}) \|^2_{L^2} d\tau
+ \mu \int_0^t \int_{\p \OM} (\mathbf{u} - \hat{\mathbf{u}}) \cdot K \cdot (\mathbf{u} - \hat{\mathbf{u}}) ds d\tau \\
& \le C \int_0^t \| \n-\hat{\n} \|^2_{L^2} \left( \| \hat{\mathbf{u}}_t \|^2_{L^4} + \| \na \hat{\mathbf{u}} \|^2_{L^4} \right) d\tau
+ C \int_0^t \| \na \hat{\mathbf{u}} \|_{L^\infty} \| \sqrt{\n} (\mathbf{u}-\hat{\mathbf{u}}) \|^2_{L^2} d\tau.
\ea\ee
On the other hand, we deduce from $(\ref{isol1})_1$ and $(\ref{isol1})_1$ that
\be\ba\la{uins11}
(\n-\hat{\n})_t +\div ( \mathbf{u} (\n-\hat{\n}) ) = (\hat{\mathbf{u}}-\mathbf{u}) \cdot \na \hat{\n}.
\ea\ee
Multiplying (\ref{uins11}) by $\n-\hat{\n}$, integrating over $\OM$, and applying Poincar\'e's inequality, we arrive at
\be\ba\nonumber
\frac{1}{2} \frac{d}{dt} \int ( \n-\hat{\n} )^2 dx
& = \int (\hat{\mathbf{u}}-\mathbf{u}) \cdot \na \hat{\n} \  (\n-\hat{\n}) dx \\
& \le C \| \hat{\mathbf{u}} - \mathbf{u} \|_{L^{\frac{2q}{q-2}}} \| \na \hat{\n} \|_{L^q} 
\| \n-\hat{\n} \|_{L^2} \\
& \le C \| \na (\hat{\mathbf{u}} - \mathbf{u}) \|_{L^2} \| \n-\hat{\n} \|_{L^2},
\ea\ee
which together with (\ref{uins9}) and H\"older's inequality gives
\be\la{uins13}\ba
\| \n-\hat{\n} \|^2_{L^2}
&\le C t \int_0^t \| \na ( \mathbf{u} - \hat{\mathbf{u}} ) \|^2_{L^2} d\tau \\
&\le C t \int_0^t \left( \| \curl(\mathbf{u} - \hat{\mathbf{u}}) \|^2_{L^2}
+ \int_{\p \OM} (\mathbf{u} - \hat{\mathbf{u}}) \cdot K \cdot (\mathbf{u} - \hat{\mathbf{u}}) dS \right) d\tau.
\ea\ee
We set
\be\ba\la{uins14}
f(t) \triangleq & \| \sqrt{\n} (\mathbf{u}-\hat{\mathbf{u}}) \|^2_{L^2}
+ \int_0^t \left( \| \sqrt{\n} (\mathbf{u}-\hat{\mathbf{u}}) \|^2_{L^2}
+ \mu \| \curl(\mathbf{u} - \hat{\mathbf{u}}) \|^2_{L^2} \right) d\tau \\
& + \mu \int_0^t \int_{\p \OM} (\mathbf{u} - \hat{\mathbf{u}}) \cdot K \cdot (\mathbf{u} - \hat{\mathbf{u}}) ds d\tau.
\ea\ee
Putting (\ref{uins13}) into (\ref{uins10}), we conclude that
\be\ba\la{uins15}
f(t) \le C \int_0^t f(\tau)
\left( 1 + \tau \| \hat{\mathbf{u}}_t \|^2_{L^4} + s\| \na \hat{\mathbf{u}} \|^2_{L^4} + \| \na \hat{\mathbf{u}} \|_{L^\infty} \right) d\tau.
\ea\ee
An application of Gr\"onwall's inequality to (\ref{uins15}), combined with (\ref{insc3}), shows that $f(t)=0$.
It then follows from (\ref{uins9}), (\ref{uins13}), and (\ref{uins14}) that
$\mathbf{u}=\hat{\mathbf{u}}$ and $\n=\hat{\n}$ a.e. in $\OM \times (0,T)$, for any $0<T<\infty$.

Moreover, $(\ref{isol1})_2$ implies that $\na \pi=\na \hat{\pi}$.
Thus, $(\n,\mathbf{u},\pi)$ is the unique strong solution of (\ref{isol1}) and satisfies (\ref{insc3}).
The uniqueness ensures that the entire sequence $(\n^\nu,\mathbf{u}^\nu)$ converges to the solution $(\n,\mathbf{u})$ of (\ref{isol1}),
thereby completing the proof of Theorem \ref{th01}.

Proof of Theorem \ref{th3}.
The proof of Theorem \ref{th3} follows the same arguments as those in \cite{CL}, and therefore we omit the details here.

\bigskip

\noindent\textbf{Data availability.} No data was used for the research described in the article.

\bigskip

\noindent\textbf{Conflict of interest.} The author declares no conflict of interest.


\begin {thebibliography} {99}




\bibitem{AJ}J. Aramaki, $L^p$ theory for the div-curl system,
Int. J. Math. Anal. (Ruse) {\bf 8} (2014), no.~5-8, 259--271.

\bibitem{BKM}J.~T. Beale, T. Kato and A.~J. Majda,
Remarks on the breakdown of smooth solutions for the $3$-D Euler equations,
Comm. Math. Phys. {\bf 94} (1984), no.~1, 61--66.



%\bibitem{BL}{\sc J. Bergh and J. L\"ofstr\"om}, 
%{\em Interpolation spaces. An introduction}, 
%Grundlehren der Mathematischen Wissenschaften, No. 223, Springer, Berlin-New York, 1976.

%\bibitem{BW}{\sc H.~R. Brezis and S. Wainger}, 
%{\em A note on limiting cases of Sobolev embeddings and convolution inequalities}, 
%Comm. Partial Differential Equations. {\bf 5} (1980), no.~7, 773--789.

\bibitem{CL}G.~C. Cai and J. Li,
Existence and exponential growth of global classical solutions to the compressible Navier-Stokes equations with slip boundary conditions in 3D bounded domains,
Indiana Univ. Math. J. {\bf 72} (2023), no.~6, 2491--2546.


\bibitem{CCK}Y. Cho, H.~J. Choe and H. Kim,
Unique solvability of the initial boundary value problems for compressible viscous fluids,
J. Math. Pures Appl. (9) {\bf 83} (2004), no.~2, 243--275.



\bibitem{CK}Y. Cho and H. Kim,
On classical solutions of the compressible Navier-Stokes equations with nonnegative initial densities,
Manuscripta Math. {\bf 120} (2006), no.~1, 91--129.



\bibitem{CK2}H.~J. Choe and H. Kim,
Strong solutions of the Navier-Stokes equations for isentropic compressible fluids,
J. Differential Equations {\bf 190} (2003), no.~2, 504--523.



%\bibitem{CRW}{\sc R.~R. Coifman, R. Rochberg and G.~L. Weiss}, 
%{\em Factorization theorems for Hardy spaces in several variables}, 
%Ann. of Math. (2) {\bf 103} (1976), no.~3, 611--635.

%\bibitem{CM}{\sc R.~R. Coifman and Y.~F. Meyer}, 
%{\em On commutators of singular integrals and bilinear singular integrals}, 
%Trans. Amer. Math. Soc. {\bf 212} (1975), 315--331.




\bibitem{DM}R. Danchin and P.~B. Mucha,
Compressible Navier-Stokes equations with ripped density,
Comm. Pure Appl. Math. {\bf 76} (2023), no.~11, 3437--3492.

% \bibitem{D}{\sc B. Desjardins}, 
% {\em Regularity of weak solutions of the compressible isentropic Navier-Stokes equations}, 
% Comm. Partial Differential Equations {\bf 22} (1997), no.~5-6, 977--1008.

% \bibitem{DL} {\sc R.J. Diperna and P.L. Lions}, {\em Ordinary differential equations, transport theory and
% Sobolev spaces},
%  Invent. Math, {\bf 98} (1989), no. 3, 511--547.

%\bibitem{E}{\sc H. Engler}, 
%{\em An alternative proof of the Brezis-Wainger inequality}, 
%Comm. Partial Differential Equations {\bf 14} (1989), no.~4, 541--544.

\bibitem{EL}L.~C. Evans,
Partial differential equations, second edition,
Graduate Studies in Mathematics, 19, Amer. Math. Soc., Providence, RI, 2010.



\bibitem{FLL}X. Fan, J. X. Li and J. Li,
Global existence of strong and weak solutions to 2D compressible Navier-Stokes system in bounded domains with large data and vacuum,
Arch. Ration. Mech. Anal. {\bf 245} (2022), no.~1, 239--278.



\bibitem{FLW}X. Fan, J. Li and X. Wang,
Large-Time Behavior of the 2D Compressible Navier-Stokes System in Bounded Domains with Large Data and Vacuum,
arXiv:2310.15520.



\bibitem{F}E. Feireisl,
Dynamics of Viscous Compressible Fluids, Oxford Lecture Series in Mathematics and its Applications vol. 26, Oxford University Press, Oxford, 2004.



\bibitem{FNP}E. Feireisl, A. Novotn\'y{} and H. Petzeltov\'a,
On the existence of globally defined weak solutions to the Navier-Stokes equations,
J. Math. Fluid Mech. {\bf 3} (2001), no.~4, 358--392.



%\bibitem{FP}{\sc E. Feireisl and H. Petzeltov\'a}, 
%{\em Large-time behaviour of solutions to the Navier-Stokes equations of compressible flow}, 
%Arch. Ration. Mech. Anal. {\bf 150} (1999), no.~1, 77--96.


% \bibitem{G} {\sc G.P. Galdi}, {\em An Introduction to the Mathematical Theory of the Navier-Stokes Equations. Vol. I:
% Linearized Steady Problems}, Springer Tracts in Natural Philosophy, vol. 38, Springer-Verlag, New
% York, 1994.



\bibitem{Ge}P. Germain,
Weak-strong uniqueness for the isentropic compressible Navier-Stokes system,
J. Math. Fluid Mech. {\bf 13} (2011), no.~1, 137--146.



\bibitem{GWX}Z. Guo, Y. Wang and C. Xie,
Global strong solutions to the inhomogeneous incompressible Navier-Stokes system in the exterior of a cylinder,
SIAM J. Math. Anal. {\bf 53} (2021), no.~6, 6804--6821.



% \bibitem{H4}{\sc D. Hoff}, 
% {\em Global existence for 1D, compressible, 
% isentropic Navier-Stokes equations with large initial data}, Trans. Amer. Math. Soc. {\bf 303} (1987), no.~1, 169--181.



\bibitem{H1}D. Hoff,
Global solutions of the Navier-Stokes equations for multidimensional compressible flow with discontinuous initial data,
J. Differential Equations {\bf 120} (1995), no.~1, 215--254.



\bibitem{H2}D. Hoff,
Strong convergence to global solutions for multidimensional flows of compressible, viscous fluids with polytropic equations of state and discontinuous initial data,
Arch. Rational Mech. Anal. {\bf 132} (1995), no.~1, 1--14.



%\bibitem{H5}{\sc D. Hoff}, 
%{\em  Dynamics of singularity surfaces for compressible, viscous flows in two space dimensions}, 
%Comm. Pure Appl. Math. {\bf 55} (2002), no.~11, 1365--1407.



\bibitem{H3}D. Hoff,
Compressible flow in a half-space with Navier boundary conditions,
J. Math. Fluid Mech. {\bf 7} (2005), no.~3, 315--338.



\bibitem{H2021}X.-D. Huang,
On local strong and classical solutions to the three-dimensional barotropic compressible Navier-Stokes equations with vacuum,
Sci. China Math. {\bf 64} (2021), no.~8, 1771--1788.



% \bibitem{HL2}{\sc X.-D. Huang and J. Li},
%  {\em Existence and blowup behavior of global strong solutions to the two-dimensional barotrpic compressible Navier-Stokes system with vacuum and large initial data},
% J. Math. Pures Appl. (9) {\bf 106} (2016), no.~1, 123--154.

% \bibitem{HL}{\sc X.-D. Huang and J. Li}, 
% {\em Global classical and weak solutions to the three-dimensional full compressible Navier-Stokes system with vacuum and large oscillations},
%  Arch. Ration. Mech. Anal. {\bf 227} (2018), no.~3, 995--1059.

% \bibitem{HW}{\sc X.-D. Huang and Y. Wang}, 
% {\em Global strong solution to the 2D nonhomogeneous incompressible MHD system},
% J. Differential Equations {\bf 254} (2013), no.~2, 511--527.

%\bibitem{HFLX3}{\sc F. Huang, J. Li and Z. Xin}, 
%{\em Convergence to equilibria and blowup behavior of global strong solutions to the Stokes approximation equations for two-dimensional compressible flows with large data}, 
%J. Math. Pures Appl. (9) {\bf 86} (2006), no.~6, 471--491.


% \bibitem{HLX3}{\sc X.-D. Huang, J. Li and Z. Xin}, 
%  {\em Blowup criterion for viscous baratropic flows with vacuum states}, 
%  Comm. Math. Phys. {\bf 301} (2011), no.~1, 23--35.



\bibitem{HLX1}X.-D. Huang, J. Li and Z. Xin,
Serrin-type criterion for the three-dimensional viscous compressible flows,
SIAM J. Math. Anal. {\bf 43} (2011), no.~4, 1872--1886.



\bibitem{HLX2}X.-D. Huang, J. Li and Z. Xin,
Global well-posedness of classical solutions with large oscillations and vacuum to the three-dimensional isentropic compressible Navier-Stokes equations,
Comm. Pure Appl. Math. {\bf 65} (2012), no.~4, 549--585.



% \bibitem{HSYY}{\sc X.-D. Huang et al.}, 
% {\em Global large strong solutions to the radially symmetric compressible Navier-Stokes equations in 2D solid balls}, 
% J. Differential Equations {\bf 396} (2024), 393--429.



% \bibitem{JZ}{\sc S. Jiang, P. Zhang}, {\em On spherically symmetric solutions of the compressible isentropic Navier-Stokes equations},
%  Comm. Math. Phys. {\bf 215} (2006), no.~3, 559--581.



\bibitem{K}T. Kato,
Remarks on the Euler and Navier-Stokes equations in ${\bf R}^2$,
Proc. Sympos. Pure Math., {\bf 45}, (1986),1--7.



% \bibitem{KS}{\sc A.~V. Kazhikhov and V.~V. Shelukhin},
% {\em Unique global solution with respect to time of
% initial-boundary value problems for one-dimensional equations
% of a viscous gas}
%  Prikl. Mat. Meh. {\bf 41} (1977), no.~2J. Appl. Math. Mech. {\bf 41} (1977), no.~2.



% \bibitem{LLL}{\sc J. Li, Z. Liang}, {\em On local classical solutions to the Cauchy problem of the two-dimensional barotropic compressible Navier-Stokes equations with vacuum},
% J. Math. Pures Appl. (9) {\bf 102} (2014), no.~4, 640--671. 



\bibitem{LKV}Q. Lei,
Global Strong Solutions to the Three-Dimensional Axisymmetric Compressible Navier-Stokes Equations with Large Initial Data and Vacuum, arXiv:2509.11260.



\bibitem{LX3}Q. Lei and C. Xiong,
Global Existence and Incompressible Limit for Compressible Navier-Stokes Equations with Large Bulk Viscosity Coefficient and Large Initial Data, arXiv:2507.01432.



\bibitem{LX4}Q. Lei and C. Xiong,
Global Existence and Incompressible Limit for Compressible Navier-Stokes Equations in Bounded Domains with Large Bulk Viscosity Coefficient and Large Initial Data, arXiv:2507.02462.



\bibitem{LX}J. Li and Z. Xin,
Some uniform estimates and blowup behavior of global strong solutions to the Stokes approximation equations for two-dimensional compressible flows,
J. Differential Equations {\bf 221} (2006), no.~2, 275--308.



\bibitem{LX2}J. Li and Z. Xin,
Global well-posedness and large time asymptotic behavior of classical solutions to the compressible Navier-Stokes equations with vacuum,
Ann. PDE {\bf 5} (2019), no.~1, Paper No. 7, 37 pp.



\bibitem{LZZ}J. Li, J.~W. Zhang and J.~N. Zhao,
On the global motion of viscous compressible barotropic flows subject to large external potential forces and vacuum,
SIAM J. Math. Anal. {\bf 47} (2015), no.~2, 1121--1153.



% \bibitem{L1} {\sc P.L. Lions}, {\em Mathematical Topics in Fluid Mechanics. Vol. 1: Incompressible Models},
% Oxford Lecture Series in Mathematics and its Applications, vol. 3, The Clarendon Press, Oxford University
% Press, New York, 1996. Oxford Science Publications.



\bibitem{L2}P.L. Lions,
Mathematical Topics in Fluid Mechanics. Vol. 2: Compressible Models,
Oxford University Press, New York, 1998.




\bibitem{MN1}A. Matsumura, T. Nishida,
The initial value problem for the equations of motion of viscous and heat-conductive gases,
J. Math. Kyoto Univ. {\bf 20}(1) (1980), 67--104.



% \bibitem{MN2}{\sc A. Matsumura, T. Nishida}, {\em Initial boundary value problems for the equations of
% motion of general fluids},
% in: Proc. Fifth Int. Symp. on Computing Methods in
% Appl. Sci. and Engng. (1982), 389--406.

% \bibitem{MN3}{\sc A. Matsumura, T. Nishida}, {\em Initial value problem for the equations of motion of
% compressible viscous and heat-conductive fluids},
% Comm. Math. Phys. {\bf 89} (1983), 445--464.


% \bibitem{MD}{\sc D.~I.~R. Mitrea}, {\em Integral equation methods for div-curl problems for planar vector fields in nonsmooth domains},
% Differential Integral Equations {\bf 18} (2005), no.~9, 1039--1054.


%\bibitem{MP}{\sc A. Matsumura, M. Padula}, {\em Stability of the stationary solutions of compressible viscous fluids with large external forces},
%Stab. Appl. Anal. Cont. Media 2 (1992) 183–202.



\bibitem{N}J. Nash,
Le probl\`{e}me de Cauchy pour les \'{e}quations diff\'{e}rentielles d'un fluide g\'{e}n\'{e}ral,
Bull. Soc. Math. France {\bf 90} (1962), 487--497 (French).



%\bibitem{NI}{\sc L. Nirenberg}, {\em On elliptic partial differential equations},
% Ann. Scuola Norm. Sup. Pisa Cl. Sci. (3) {\bf 13} (1959), 115--162.

%\bibitem{NS}{\sc A. Novotn\'y{} and I. Stra\v skraba}, 
%{\em Convergence to equilibria for compressible Navier-Stokes equations with large data}, 
% Ann. Mat. Pura Appl. (4) {\bf 179} (2001), 263--287.



\bibitem{NS}A. Novotn\'y{} and I. Stra\v skraba,
Introduction to the mathematical theory of compressible flow,
Oxford Lecture Series in Mathematics and its Applications, 27, Oxford Univ. Press, Oxford, 2004.



% \bibitem{PSW}{\sc Y.~F. Peng, X. Shi and Y.~S. Wu}, 
% {\em Exponential decay for Lions-Feireisl's weak solutions to the barotropic compressible Navier-Stokes equations in 3D bounded domains}, 
% Indiana Univ. Math. J. {\bf 70} (2021), no.~5, 1813--1831.



\bibitem{STT}M.~A. Sadybekov, B.~T. Torebek and B.~K. Turmetov,
Representation of Green's function of the Neumann problem for a multi-dimensional ball,
Complex Var. Elliptic Equ. {\bf 61} (2016), no.~1, 104--123.



\bibitem{SS}R. Salvi and I. Stra\v skraba,
Global existence for viscous compressible fluids and their behavior as $t\to\infty$,
J. Fac. Sci. Univ. Tokyo Sect. IA Math. {\bf 40} (1993), no.~1, 17--51.



\bibitem{S}J. Serrin,
On the uniqueness of compressible fluid motions,
Arch. Rational Mech. Anal. {\bf 3} (1959), 271--288.



\bibitem{SES}E.~M. Stein and R. Shakarchi,
Complex analysis, Princeton Univ. Press, Princeton, NJ, 2003.







% \bibitem{S1}{\sc D. Serre},
% {\em Solutions faibles globales des \'equations de Navier-Stokes pour un fluide compressible}, 
% C. R. Acad. Sci. Paris S\'er. I Math. {\bf 303} (1986), no.~13, 639--642.


% \bibitem{S2}{\sc D. Serre},
% {\em Sur l'\'equation monodimensionnelle d'un fluide visqueux, compressible et conducteur de chaleur,} 
% C. R. Acad. Sci. Paris S\'er. I Math. {\bf 303} (1986), no.~14, 703--706.




% \bibitem{TG}{\sc G.~G. Talenti}, 
% {\em Best constant in Sobolev inequality}, 
% Ann. Mat. Pura Appl. (4) {\bf 110} (1976), 353--372.



\bibitem{WWV}W. von~Wahl,
Estimating $\nabla u$ by ${\rm div}\, u$ and ${\rm curl}\, u$,
Math. Methods Appl. Sci. {\bf 15} (1992), no.~2, 123--143.



\bibitem{WG}Y. Wang and Z.~H. Guo,
Global well-posedness and large-time behavior for inhomogeneous incompressible Navier-Stokes system in the exterior of a cylinder,
Math. Methods Appl. Sci. {\bf 48} (2025), no.~8, 8602--8616.





% \bibitem{Z}{\sc X. Zhong}, 
% {\em Global well-posedness to the Cauchy problem of two-dimensional nonhomogeneous heat conducting Navier-Stokes equations}, 
% J. Geom. Anal. {\bf 32} (2022), no.~7, Paper No. 200, 22 pp.


% \bibitem{ZAA}{\sc A.~A. Zlotnik}, 
% {\em Uniform estimates and the stabilization of symmetric solutions of
% 	a system of quasilinear equations}, 
% Differ. Equ. {\bf 36} (2000), no.~5, 701--716.



% \bibitem{VK}{\sc V.A. Va\u igant, A.V. Kazhikhov},
% {\em On the existence of global solutions of two-dimensional Navier-Stokes
% equations of a compressible viscous fluid},
%  (Russian) Sibirsk. Mat. Zh. {\bf 36} (1995), no. 6, 1283–1316, ii;
% translation in Siberian Math. J. {\bf 36} (1995), no. 6, 1108–1141.

\end {thebibliography}
\end{document}